\documentclass[reqno]{amsart}%
\usepackage{graphicx}
\usepackage{amssymb,amstext,latexsym,amsmath, amscd}
\usepackage{enumitem}
\usepackage[all]{xy}
\usepackage{color}
\usepackage{stmaryrd}
\usepackage{tensor}
\usepackage{hyperref}
\usepackage{tikz}
\usepackage{tikz-cd}
\usetikzlibrary{matrix,calc,arrows,decorations.pathreplacing, patterns}
\usepackage{pdflscape}
\usepackage{rotating}
\usepackage{ucs}
\usepackage[utf8x]{inputenc}
\newcommand{\cP}{\mathcal{P}}

\newcommand{\cS}{\mathcal{S}}

\newcommand{\C}{\mathcal{C}}
\newcommand{\F}{\mathcal{F}}            
\newcommand{\G}{\mathcal{G}}            
\renewcommand{\H}{\mathcal{H}}          

\newcommand{\M}{\mathcal{M}}
\newcommand{\N}{\mathcal{N}}

\newcommand{\I}{\mathcal{I}}
\renewcommand{\O}{\mathcal{O}}
\renewcommand{\S}{\mathcal{S}}

\newcommand{\rmap}{\longrightarrow}

\DeclareMathOperator{\Spec}{Spec}
\DeclareMathOperator{\Hom}{Hom}

\renewcommand{\dim}{\text{\rm dim}\,}  

\newcommand{\dezero}{\frac{d}{d\epsilon}_{|\epsilon=0}}

\newcommand{\arrows}{\rightrightarrows}

\newcommand{\ot}{\longleftarrow}
\newcommand{\tto}{\rightrightarrows}    
\newcommand{\x}{\times}
\newcommand{\lx}{\ltimes}

\newcommand{\Z}{\mathbb{Z}}
\newcommand{\R}{\mathbb{R}}
\newcommand{\g}{\mathfrak{g}}           
\renewcommand{\gg}{\mathfrak{g}}        
\newcommand{\m}{\mathfrak{m}}
\renewcommand{\a}{\mathfrak{a}}

\newcommand{\CCC}{\mathfrak{C}}

\newcommand{\CC}{C^\infty}

\newcommand{\CR}[1]{C^\infty(\mathbb{R}^{#1})}

\theoremstyle{remark}

\newtheorem{rmk}{Remark}
\newtheorem{remark}[rmk]{Remark}

\theoremstyle{plain}
\newtheorem{thm}{Theorem}[section]
\newtheorem{theorem}[thm]{Theorem}

\newtheorem{lemma}[thm]{Lemma}

\newtheorem{corollary}[thm]{Corollary}

\newtheorem{proposition}[thm]{Proposition}

\theoremstyle{definition}

\newtheorem{definition}[thm]{Definition}

\newtheorem{example}[thm]{Example}

\newtheorem{MT}{Main Theorem}{\bf}{\it}
\newtheorem*{MC}{Main Conclusion}{\bf}{\it}

\begin{document}
\title{Orbispaces as differentiable stratified spaces}


\author{Marius Crainic}
\address{Depart of Math., Utrecht University, 3508 TA Utrecht,
The Netherlands}
\email{crainic@math.uu.nl}

\author{Jo\~ao Nuno Mestre}
\address{CMUC, Department of Mathematics, University of Coimbra,  Portugal}
\curraddr{MPIM, Vivatsgasse 7, 53111 Bonn, Germany}
\email{jnmestre@gmail.com}

\thanks{The authors would like to thank Rui Loja Fernandes, David Martinez Torres and Hessel Posthuma for useful discussions related to this paper.\newline
The first author was supported by the Nederlandse Organisatie voor Wetenschappelijk Onderzoek - Vici grant no. 639.033.312. 
The second author was supported by Funda\c{c}\~ao para a Ci\^encia e Tecnologia grant  SFRH/BD/71257/2010 under the POPH/FSE programmes and by the Centre for Mathematics of the University
of Coimbra - UID/MAT/00324/2013, funded by the Portuguese Government through
FCT/MEC and co-funded by the European Regional Development Fund through the Partnership
Agreement PT2020.}


\begin{abstract} We present some features of the smooth structure, and of the canonical stratification on the orbit space of a proper Lie groupoid. One of the main features is that of Morita invariance of these structures - it allows us to talk about the canonical structure of differentiable stratified space on the orbispace (an object analogous to a separated stack in algebraic geometry) presented by the proper Lie groupoid. 

The canonical smooth structure on an orbispace is studied mainly via Spallek's framework of differentiable spaces, and two alternative frameworks are then presented. For the canonical stratification on an orbispace, we extend the similar theory coming from proper Lie group actions. We make no claim to originality. The goal of these notes is simply to give a complementary exposition to those available, and to clarify some subtle points where the literature can sometimes be confusing, even in the classical case of proper Lie group actions.

\end{abstract}

\maketitle

\setcounter{tocdepth}{1}
\tableofcontents

\section{Introduction}\label{sec:introduction}

Lie groupoids are geometric objects that generalize Lie groups and smooth manifolds, and permit a unified approach to the study of several objects of interest in differential geometry, such as Lie group actions, foliations and principal bundles (see for example \cite{cannas,mackenzie,ieke_mrcun} and the references therein). They have found wide use in Poisson and Dirac geometry (e.g. \cite{GS}) and noncommutative geometry \cite{Connes}.

One of the main features of Lie groupoids is that they permit studying singular objects, in particular quotients, as if they were smooth. This is because groupoids also generalize equivalence relations on manifolds, but keep record of the different ways in which points can be equivalent. As such, every Lie groupoid $\G\arrows M$ has an associated quotient space $M/\G$ by the equivalence relation it defines - its orbit space. Due to their unifying nature, Lie groupoids are then useful in studying a variety of singular spaces such as leaf spaces of foliations, orbit spaces of actions, and orbifolds.

As a topological space, the orbit space $M/\G$ of a Lie groupoid can be very uninteresting, so when talking about its orbit space, it is common to consider extra structure,  constructed from the Lie groupoid.
There are several different approaches to building a good model for singular, or quotient spaces. To name a few, one can build Haefliger's classifying space \cite{haefliger2} (also related to Van Est's $S$-atlases \cite{VE-atlas}), or a noncommutative algebra of functions as a model for the quotient, as in the Noncommutative geometry of Connes \cite{Connes}. There is also the theory of stacks, which grew out of algebraic geometry \cite{Artin,SGA,Deligne-Mumford,Giraud} (where algebraic groupoids are used instead) and has recently gained increasing interest in the context of differential geometry (cf. e.g. \cite{behrend_xu,heinloth,lerman-stacks,metzler}). The common feature to all these approaches is that they start by modelling the situation at hand by the appropriate Lie groupoid (the holonomy groupoid of a foliation \cite{haefliger2}, or an orbifold groupoid as an orbifold atlas \cite{pronk}, for example). The way in which a good model $M//\G$ for the quotient is then constructed out of the Lie groupoid is what differs.
We refer to \cite{Moerd-Top} for a comparison of these and some other approaches to modelling quotients via groupoids.

Let us mention in more detail the role of Lie groupoids in studying orbifolds. The idea is that these are spaces which are locally modelled by quotients of Euclidean space by actions of finite groups (where the actions are considered part of the structure!). Simple as it may seem, making this description precise has been subtle. Manifestations of orbifolds were first considered in classical algebraic geometry, where they were called by the name of varieties with quotient singularities (See for example \cite{Ruan} and the references therein for an account). But these consisted only of the underlying quotient space. Taking into account the group actions, classical orbifolds were studied by Satake \cite{Satake} and Thurston \cite{Thurston}, and were defined in terms of charts and atlases, similar to manifold atlases. However, this definition still had serious limitations. 

The modern take on orbifolds, following the work of Haefliger \cite{haefliger2} and Moerdijk-Pronk \cite{pronk}, bypasses those limitations by using a special class of Lie groupoids to describe orbifolds - that of proper \'etale Lie groupoids. In modern terminology, an orbifold atlas on a topological space $X$ is a proper \'etale Lie groupoid $\G\arrows M$, together with an homeomorphism between its orbit space $M/ \G$ and $X$. An orbifold structure on $X$ is then defined as an equivalence class of orbifold atlases on $X$. The correct notion of equivalence between atlases is provided by the notion of Morita equivalence between the corresponding Lie groupoids. This makes sense because, in particular, a Morita equivalence between Lie groupoids induces an homeomorphism between their orbit spaces. Indeed, Morita invariant information encodes the transverse geometry of a Lie groupoid, meaning that geometry of its orbit space which is independent of a particular choice of atlas (see Theorem 4.3.1 in \cite{matias} for a precise statement of this fact).

The definition of orbifolds using Lie groupoids may seem more complicated than the one in terms of local charts, but it has some crucial advantages. First of all, it allows naturally for the treatment of orbifolds where the local actions are non-effective. This is a situation which occurs in several important examples, such as sub-orbifolds, weighted projective spaces, or simple moduli spaces. It also makes possible to correctly deal with morphisms of orbifolds. Some standard textbook references on orbifolds are \cite{Ruan,ieke_mrcun}.

In a similar fashion, we can describe atlases for differentiable stacks if we now allow for the use of general Lie groupoids. Hence, studying the differential geometry of a differentiable stack can essentially be done by studying Morita invariant geometry on Lie groupoids.

There are some very recent developments in the treatment of the differential geometry of singular spaces in this way. Let us name a few examples (which certainly do not constitute a complete list). Vector fields on differentiable stacks have been studied via multiplicative vector fields on Lie groupoids \cite{lerman2,hepworth,lerman,ortiz-waldron}. There are now Riemannian metrics for differentiable stacks, studied via Riemannian groupoids \cite{matias_rui_metrics,matias_rui_fibrations}. Integral affine structures on orbifolds appear via (special classes of) symplectic groupoids \cite{PMCT1,PMCT2,PMCT3}. Measures and densities on differentiable stacks can be studied via transverse measures for Lie groupoids \cite{measures_stacks,alan-volume}.

In these notes we focus on two aspects of the geometry of a particular class of differentiable stacks, called orbispaces. Those are the differentiable stacks which have an atlas given by a proper Lie groupoid. 
We describe and present some features of the smooth structure, and of the canonical stratification on an orbispace (cf. \cite{hessel}) which we get associated to any Lie groupoid presenting it.

Orbispaces generalize orbifolds and have particularly good features among differentiable stacks. Among them, let us mention that a linearization result for proper Lie groupoids \cite{ivan,alan,zung} permits the analysis of the local geometry of orbispaces. Essentially, the normal form given by the linearization provides adapted coordinates to the structure of the orbispace. As in the case of orbifolds, early definitions for orbispaces have appeared in terms of charts \cite{Chen,Pflaum2,PHD-Schwarz}, but nowadays they are usually treated using the language of Lie groupoids (cf. e.g. \cite{matias,thesis,Giorgio}).

In order to study the smooth structure of an orbispace, we take an approach inspired by algebraic geometry and see the orbispace as a locally ringed space, equipped with a sheaf of smooth functions. As is the general philosophy, the sheaf of smooth functions on the orbispace can be described in terms of smooth invariant functions on any Lie groupoid presenting it. In this way we frame orbispaces in the theory of differentiable spaces, which is a simple version of scheme theory for differential geometry \cite{juan}. We then present two alternative ways to describe the smooth structure, each with its own advantages. A different strategy, which also proves useful, but that we shall not discuss in this text is that of equipping an orbispace with the structure of a diffeology. We refer to \cite{Watts3,Watts,Watts4} for details on this approach.

The canonical stratification of an orbispace is a decomposition of the orbispace into subspaces which carry a smooth manifold structure and which fit together nicely. When studying an orbispace, there is a canonical stratification which appears. It is closely connected to the stratification induced by the partition by orbit types in the theory of proper Lie group actions.

We explain how to extend the constructions of the canonical stratification of a proper Lie group action into the context of proper Lie groupoids and orbispaces (cf. e.g. \cite{DK,Pflaum}). This gives a different (but equivalent) take on the stratification of \cite{hessel}. More specifically, we use the description of stratifications via partitions by manifolds (cf. e.g. \cite{DK}) instead of the approach using germs of decompositions (cf. e.g. \cite{Pflaum,hessel}).

We also present proofs for some results that seem to be commonly accepted as an extension of the theory of orbit type stratifications for proper Lie group actions, but which, to the best of our knowledge, were not readily available. This is the case for example with the principal Morita type theorem (\ref{thm-principal_type}), which states the existence of a connected open dense stratum of the canonical stratification of an orbispace by Morita type.

\paragraph{\textbf{Outline of the paper and of the main results}}

In section 2 we present some background material on Lie groupoids, including some basics on Morita equivalences, proper Lie groupoids, and orbispaces.

In section 3 we describe ``the canonical smooth structure'' that orbispaces are endowed with. Several approaches to smooth structures on singular spaces will be recalled in the paper. We focus on the framework provided by differentiable spaces (cf. \cite{juan}) and prove Main theorem 1. We then move to other settings and derive two variations of the main theorem 1. 

\begin{MT}[\textbf{and variations}]
The orbit space $X$ of a proper Lie groupoid $\G\arrows M$, together with the sheaf $\C^\infty_X$ on $X$ (Definition \ref{dfn-functions-on-orbispace}), is a reduced differentiable space, a locally fair affine $C^\infty$-scheme, and a subcartesian space.

These smooth structures are Morita invariant, thus associated to the orbispace presented by $\G$.
\end{MT}

More refined versions of this statement can be found inside the paper (Theorem \ref{Xissmooth} and Propositions \ref{Prop-Var1} and \ref{Prop-Var2}).

In section 4 we move to the canonical stratification on orbispaces. Its description is inspired by the similar stratifications for proper Lie group actions (cf. e.g. \cite{DK}), that it generalizes, but adapted to the groupoid context. The main idea here is that of ``Morita types'' - the pieces of the partition giving rise to the canonical stratification which, by construction, will be Morita invariant.

\begin{MT}
Let $\G\arrows M$ be a proper Lie groupoid. Then the partitions of $M$ and of $X=M/\G$ by connected components of Morita types are stratifications. 

Given a Morita equivalence between two Lie groupoids $\G$ and $\H$, the induced homeomorphism at the level of orbit spaces preserves this stratification.
\end{MT}

We will also prove a principal type theorem for the canonical stratification on the orbit space (Theorem \ref{thm-principal_type}).

In section 5 we combine the previous two sections, looking at the interplay between the smooth structure and the canonical stratification on orbispaces.

\begin{MC}
Let $\G\arrows M$ be a proper Lie groupoid. Then $M$ and the orbit space $X=M/\G$, together with the canonical stratifications, are differentiable stratified spaces. Moreover, the canonical stratifications of $M$ and $X$ are Whitney stratifications.

Any Morita equivalence between two proper Lie groupoids induces an isomorphism of differentiable stratified spaces between their orbit spaces.
\end{MC}

\newpage

\section{Background}\label{sec:preliminaries}

\subsection{Lie groupoids}\label{sec-liegpd}

Recall that a \textbf{Lie groupoid} consists of two smooth manifolds, $\G$ and $M$, called the space of arrows and the space of objects respectively, together with submersions $s,t:\G\to M$, called the \textbf{source} and \textbf{target} respectively, a partially defined \textbf{multiplication} $m:\G^{(2)}\to\G$ (defined on the \textbf{space of composable arrows} $\G^{(2)}=\{(g,h)\in\G\ |\ s(g)=t(h)\}$), a \textbf{unit} section $u:M\to \G$, and an \textbf{inversion} $i:\G\to\G$, satisfying group-like axioms (see e.g. \cite{ieke_mrcun}).

We will also use the notations $u(x)=1_x$, $i(g)=g^{-1}$ and $m(g,h)=gh$. 
An arrow $g$ with source $x$ and target $y$ is sometimes denoted more graphically by $g:x\to y$, $x\stackrel{g}{\to} y$ or $y\stackrel{g}{\ot} x$ and we commonly denote the groupoid $\G$ over $M$ by $\G \arrows M$.

The space of arrows $\G$ is not required to be Hausdorff, but the space of objects $M$ and the fibres of the source map $s:\G\to M$ are. This is done in order to accommodate several natural examples of groupoids for which the space of arrows may fail to be Hausdorff. A typical source of such examples is foliation theory.

From the definition of a Lie groupoid $\G\arrows M$ we can conclude that the inversion map is a diffeomorphism of $\G$ and that the unit map is an embedding $u:M \hookrightarrow \G$. We often identify the base of a groupoid with its image by the unit embedding.

\begin{example} \ \
\begin{itemize}
\item[1.] (Lie groups) Any Lie group $G$ can be seen as a Lie groupoid over a point $G\arrows \{*\}$.




\item[2.] (Submersion groupoids) Given any submersion $\pi: M\to B$ there is a groupoid $M\x_\pi M\arrows M$, for which the arrows are the pairs $(x,y)$ such that $\pi(x)=\pi(y)$, and the structure maps are defined by $s(x,y)=y$, $t(x,y)=x$ and $(x,y)\cdot(y,z)=(x,z)$. This is called the submersion groupoid of $\pi$ and it is sometimes denoted by $\G(\pi)$. In the particular case of $\pi$ being the identity map of $M$ we obtain the so-called unit groupoid; when $B$ is a point, we obtain the pair groupoid of $M$.

\item[3.] (Action groupoids) Let $G$ be a Lie group acting smoothly on a manifold $M$. Then we can form the action Lie groupoid $G\lx M\arrows M$. The objects are the points of $M$, and the arrows are pairs $(g,x)\in G\x M$. The structure maps are defined by $s(g,x)=x$, $t(g,x)=g\cdot x$, $1_x=(e,x)$, $(g,x)^{-1}=(g^{-1},g\cdot x)$ and $(g,h\cdot x)(h,x)=(gh, x)$.

\item[4.] (Gauge groupoids) Let $P\to M$ be a principal bundle with structure group $G$. Then we can take the quotient $(P\x P)/G$ of the pair groupoid of $P$ by the  diagonal action of $G$, to obtain a Lie groupoid over $M$ called the gauge groupoid of $P$ and denoted by $Gauge(P)\arrows M$.


\item[5.] (Tangent groupoids)  For any Lie groupoid $\G\arrows M$, applying the tangent functor gives us a groupoid $T\G\arrows TM$, which we call the tangent groupoid of $\G$, for which the structural maps are the differential of the structural maps of $\G$.

\end{itemize}
\end{example}

\begin{definition}
Let $\G\arrows M$ be a Lie groupoid and $x\in M$. The subsets  $s^{-1}(x)$ and $t^{-1}(x)$ of $\G$ are called the \textbf{source-fibre} of $x$ and the \textbf{target-fibre} of $x$ respectively (or $s$-fibre and $t$-fibre).
The subset $\G_x:=\{g\in \G\ |\ s(g)=t(g)=x\}\subset \G$ is called the \textbf{isotropy group} of $x$.
\end{definition}

\begin{definition}
Any Lie groupoid $\G\arrows M$ defines an equivalence relation on $M$ such that two points $x$ and $y$ are related if and only if there is an arrow $g\in\G$ such that $s(g)=x$ and $t(g)=y$. The equivalence classes are called the \textbf{orbits} of the groupoid and the orbit of a point $x\in M$ is denoted by $\O_x$. A subset of $M$ is said to be \textbf{invariant} if it is a union of orbits. Given a subset $U$ of $M$, the \textbf{saturation} of $U$, denoted by $\langle U\rangle$, is the smallest invariant subset of $M$ containing $U$.

The quotient of $M$ by this relation, endowed with the quotient topology, is called the \textbf{orbit space} of $\G$ and is denoted by $M/\G$.  
\end{definition}

The following result describes these pieces of a groupoid (cf. e.g. \cite{ieke_mrcun}).

\begin{proposition}[Structure of Lie groupoids]\label{prop-fund-str-gpd}
Let $\G\arrows M$ be a Lie groupoid and $x,y\in M$. Then:
\begin{enumerate}

\item[1.] the set of arrows from $x$ to $y$, $s^{-1}(x)\cap t^{-1}(y)$ is a Hausdorff submanifold of $\G$;
\item[2.] the isotropy group $\G_x$ is a Lie group;
\item[3.] the orbit $\O_x$ through $x$ is an immersed submanifold of $M$; 
\item[4.] the $s$-fibre of $x$ is a principal $\G_x$-bundle over $\O_x$, with projection the target map $t$.

\end{enumerate}
\end{proposition}

The partition of the manifolds into connected components of the orbits forms a foliation, which is possibly singular, in the sense that different leaves might have different dimension. To give an idea of some different kinds of singular foliations that might occur let us look at two very simple examples coming from group actions.

\begin{example}Let the circle $S^1$ act on the plane $\mathbb{R}^2$ by rotations. Then the leaves of the singular foliation on the plane corresponding to the associated action groupoid are the orbits, i.e., the origin and the concentric circles centred on it.

Let now $(\mathbb{R}_+,\x)$ act on the plane $\mathbb{R}^2$ by scalar multiplication. The leaves of the corresponding singular foliation are the origin and the radial open half-lines. 

Note that the first example has a Hausdorff orbit space; in the second example, on the other hand, there is a point in the orbit space which is dense, defined by the orbit consisting of the origin.
\end{example}

\begin{definition}A \textbf{Lie groupoid morphism} between $\G\arrows M$ and $\H\arrows N$ is a smooth functor, i.e., a pair of smooth maps $\Phi:\G\to \H$ and $\phi:M\to N$ commuting with all the structure maps. An \textbf{isomorphism} is an invertible Lie groupoid morphism.
\end{definition}








\subsection{Actions and representations}

A groupoid can act on a space fibred  over its base, with an arrow $g:x\to y$ mapping the fibre over $x$ onto the fibre over $y$. 

\begin{definition}
Let $\G\arrows M$ be a Lie groupoid and consider a surjective smooth map $\mu: P\to M$. A \textbf{(left) action} of $\G$ on $P$ along the map $\mu$, which is called the \textbf{moment map}, is a smooth map $$\G\x_M P = \{(g,p)\in \G\x P\ |\ s(g)=\mu(p)\}\to P,$$ denoted by $(g,p)\mapsto g\cdot p=gp$, such that $\mu(gp)=t(g)$, and satisfying the usual action axioms $(gh)p=g(hp)$ and $1_{\mu(p)}p=p$. We then say that $P$ is a \textbf{left $\G$-space}.
\[
\begin{tikzcd}\G \arrow[xshift=0.7ex]{d} \arrow[draw=none]{r}{\curvearrowright}  \arrow[xshift=-0.7ex]{d} & P \arrow{ld}{\mu}\\
M & 
\end{tikzcd}
\]
\end{definition}

\begin{example} Any Lie groupoid $\G\arrows M$ acts canonically on its base, with moment map the identity on $M$, by letting $g:x\to y$ act by $gx=y$; it also acts on $\G$ itself by left translations, with the target map $t:\G\to M$ as moment map, and action $g\cdot h=gh$.
\end{example}

\begin{definition}
Let $\G\arrows M$ be a Lie groupoid. A \textbf{representation} of $\G$ is a vector bundle $E$ over $M$, together with a linear action of $\G$ on $E$, meaning that for each arrow $g:x\to y$, the induced map $g:E_x\to E_y$ is a linear isomorphism.  
\end{definition}

In general, given a groupoid $\G$, there might not be many interesting representations, so let us focus on some particular classes that have natural examples.

\begin{definition} (Regular and transitive groupoids)
A Lie groupoid is called \textbf{regular} if all the orbits have the same dimension. It is called \textbf{transitive} if it has only one orbit. 
\end{definition}

\begin{example} The gauge groupoid $Gauge(P)$ of a principal bundle $P$ is transitive. Conversely, if $\G\arrows M$ is a transitive groupoid, then $\G$ is isomorphic to $Gauge(s^{-1}(x))$, the gauge groupoid of the $\G_x$-principal bundle $s^{-1}(x)\stackrel{t}{\to} M$, for any object $x\in M$.
\end{example}

\begin{example}[Representations of regular groupoids] Let $\G\arrows M$ be a regular Lie groupoid, with Lie algebroid $A$. Then $\G$ has natural representations on the kernel of the anchor map of $A$, denoted by $\mathfrak{i}$, and on the normal bundle to the orbits (which is the cokernel of the anchor), denoted by $\nu$. An arrow $g\in \G$ acts on $\alpha\in \mathfrak{i}_{s(g)}$ by conjugation, $$g\cdot\alpha=dR_{g^{-1}}\circ dL_{g}\alpha$$ and it acts on $[v]\in \nu_{s(g)}$ by the so called \textbf{normal representation}: if $g(\epsilon)$ is a curve on $\G$ with $g(0)=g$ such that $[v]=\left[\dezero s(g(\epsilon)) \right]$, then $$g\cdot[v]=\left[\dezero t(g(\epsilon)) \right].$$ In other words, $g\cdot [v]$ can be defined as $[dt(X)]$, where $X\in T_g\G$ is any $s$-lift of $v$, meaning that $ds(X)=v$.
\end{example}

\begin{example}[Restriction to an orbit]\label{normalrep} If $\G\arrows M$ is any Lie groupoid, not necessarily regular, then the normal spaces to the orbits may no longer form a vector bundle. Nonetheless, we can still get a representation of an appropriate restriction of $\G$ on some appropriate normal bundle. To be precise, if $\O$ is an orbit of $\G$, then the restriction $$\G_\O=\{g\in\G\  |\ s(g),t(g)\in\O\}$$ is a Lie groupoid over $\O$ (isomorphic to the gauge groupoid of the $\G_x$-principal bundle $s^{-1}(x)\stackrel{t}{\to} \O$, for any object $x\in \O$). It has a natural representation on $\N\O$, the normal bundle to the orbit inside of $M$, defined in the following way. Denote by $\N_x:=T_xM/T_x\O_x$ the fibre of $\N\O$ at $x$. Let $g\in \G_\O$ and $[v]\in \N_{s(g)}$.  Then, just as in the regular case, we define $g\cdot [v]$ to be $[dt(X)]$, for any $s$-lift $X\in T_g\G$ of $v$. Furthermore given any point $x$ in the orbit $\O$ we can restrict this representation to a representation of the isotropy Lie group $\G_x$ on the normal space $\N_x$,  also called the \textbf{normal representation} (or isotropy representation) of $\G_x$.
\end{example}

\subsection{Morita equivalences}

\begin{definition}
A \textbf{left $\G$-bundle} is a left $\G$-space $P$ together with a $\G$-invariant surjective submersion $\pi: P\to B$. 
A left $\G$-bundle is called \textbf{principal} if the map $\G\x_M P\to P\x_\pi P$, $(g,p)\mapsto (gp,p)$ is a diffeomorphism. So for a principal $\G$-bundle, each fibre of $\pi$ is an orbit of the $\G$-action and all the stabilizers of the action are trivial.
\[
\begin{tikzcd}\G \arrow[xshift=0.7ex]{d} \arrow[draw=none]{r}{\curvearrowright}  \arrow[xshift=-0.7ex]{d} & P \arrow{ld}\arrow{d}{\pi} \\
M & B
\end{tikzcd}
\]
\end{definition}

The notions of right action and right principal $\G$-bundle are defined in an analogous way.

\begin{definition}\label{dfn-bibundle-morita} A \textbf{Morita equivalence} between two Lie groupoids $\G\arrows M$ and $\H\arrows N$ is given by a \textbf{principal $\G-\H$-bibundle}, 
i.e., a manifold $P$ together with moment maps $\alpha: P\to M$ and $\beta: P\to N$, such that $\beta: P\to N$ is a left principal $\G$-bundle, $\alpha: P\to M$ is a right principal $\H$-bundle and the two actions commute: $g\cdot(p\cdot h)=(g\cdot p)\cdot h$ for any $g\in \G,\  p\in P$ and $h\in \H$. We say that $P$ is a bibundle realising the Morita equivalence.
\[
\begin{tikzcd}\G \arrow[xshift=0.7ex]{d} \arrow[draw=none]{r}{\curvearrowright}  \arrow[xshift=-0.7ex]{d} & P \arrow{ld}{\alpha}\arrow{rd}[swap]{\beta} \arrow[draw=none]{r}{\curvearrowleft} & \H\arrow[xshift=0.7ex]{d} \arrow[xshift=-0.7ex]{d} \\
M & & N
\end{tikzcd}
\]
\end{definition}

\begin{example}[Isomorphisms] If $f: \G \to \H$ is an isomorphism of Lie groupoids, then $\G$ and $\H$ are Morita equivalent. A bibundle can be given by the graph $Graph(f)\subset \G\times\H$, with moment maps $t\circ pr_1$ and $s\circ pr_2$, and the natural actions induced by the multiplications of $\G$ and $\H$. 
\end{example}

\begin{example}[Pullback groupoids] Let $\G$ be a Lie groupoid over $M$ and let $\alpha: P\to M$ be a surjective submersion. Then we can form the \textit{pullback groupoid} $\alpha^*\G\arrows P$, that has as space of arrows $P\times_M \G\times_M P$, meaning that arrows are triples $(p,g,q)$ with $\alpha(p)=t(g)$ and $s(g)=\alpha(q)$. The structure maps are determined by $s(p,g,q)=q$, $t(p,g,q)=p$ and  $(p,g_1,q)(q,g_2,r)=(p,g_1g_2,r)$.

The groupoids $\G$ and $\alpha^*\G$ are Morita equivalent, a bibundle being given by $\G\times_M P$. The left action of $\G$ has moment map $t\circ pr_1:\G\times_M P \to M$ is given by $g\cdot(h,p)=(gh,p)$ and the right action of $\alpha^*\G$  has moment map $pr_2:\G\times_M P\to P$ and is given by $(h,p)\cdot(p,k,q)=(hk,q)$.
\end{example}

\begin{remark}[Decomposing Morita equivalences]\label{simplemorita}

Let $\G$ and $\H$ be Morita equivalent, with bibundle $P$ as above. Since $P$ is a principal bibundle, it is easy to check that $$\alpha^* \G\, =\, P\times_M \G\times_M P\, \cong\,  P\times_M P\times_N P\, \cong\,  P\times_N \H\times_N P\,  =\, \beta^*\H,$$ as Lie groupoids over $P$.

This means that we can break any Morita equivalence between $\G$ and $\H$, using a bibundle $P$, into a chain of simpler Morita equivalences: $\G$ is Morita equivalent to $\alpha^*\G\cong\beta^*\H$, which is Morita equivalent to $\H$.
\end{remark}
\newpage

\begin{example}\label{moritaex1} \ \
\begin{enumerate}
\item[1.] Two Lie groups are Morita equivalent if and only if they are isomorphic.

\item[2.] Any transitive Lie groupoid $\G$ is Morita equivalent to the isotropy group $\G_x$ of any point  $x$ in the base.

\item[3.] Let $\G\arrows M$ be a Lie groupoid, let $N\subset M$ be a submanifold that intersects transversely every orbit it meets and let $\langle N\rangle$ denote the saturation of $N$. Then  $\G_N\arrows N$ is Morita equivalent to $\G_{\langle  N\rangle}\arrows \langle N\rangle$. As a particular case, we can take $N$ to be any open subset of $M$.

\item[4.] The groupoid $\G(\pi)$ associated to a submersion $\pi: M\to N$ is Morita equivalent to the unit groupoid $\pi(M)$. 
\end{enumerate}
\end{example}

\begin{lemma}[Morita equivalences preserve transverse geometry]\label{lemma-moritarmk} Let $\G\arrows M$ and $\H\arrows N$ be Morita equivalent Lie groupoids and let $P$ be a bibundle realising the equivalence. Then $P$ induces:
\begin{enumerate} \item[1.] A homeomorphism between the orbit spaces of $\G$ and $\H$,
\[\Phi: M/\G \rmap N/\H;\]
\item[2.] isomorphisms $\phi:\G_x\rmap \H_y$ between the isotropy groups at any points $x\in M$ and $y\in N$ whose orbits are related by $\Phi$, i.e., for which $\Phi(\O_x)=\O_y$;

\item[3.] isomorphisms $\tilde{\phi}:\N_x\rmap \N_y$ between the normal representations at any points $x$ and $y$ in the same conditions as in point 2, which are compatible with the isomorphism $\phi:\G_x\rmap \H_y$.
\end{enumerate}
\end{lemma}

\begin{proof}
First, let us define the map $\Phi$: fix a point $x$ in the base of $\G$. Then for any point $x'$ in the orbit of $x$, the fibre $\alpha^{-1}(x')$ is a single orbit for the $\H$-action on $P$; it projects via $\beta$ to a unique orbit of $\H$, which we define to be $\Phi(\O_x)$. Invariance of $\beta$ under the action of $\G$ implies that $\Phi(\O_x)$ does not depend on the choice of $x'$, so $\Phi$ is well defined. In order to see that $\Phi$ is a homeomorphism, note that bi-invariant open sets on $P$ correspond to invariant opens on $\G$ and $\H$.

Let $x\in M$ and $y\in N$ be points such that their orbits are related by $\Phi$. Then there is a $p\in P$ such that $\alpha(p)=x$ and $\beta(p)=y$. Any such $p$ induces an isomorphism $\phi_p: \G_x\to \H_y$ between the isotropy groups at $x$ and $y$, uniquely determined by the condition $gp=p\phi_p(g)$.

Moreover, $p$ induces an isomorphism $\tilde{\phi}_p: \N_x\to\N_y$ between the normal representations at $x$ and $y$, uniquely determined by \[\tilde{\phi}([(d\alpha)_p (X_p)])=[(d\beta)_p (X_p)],\ \  \mathrm{for\  all}\ X_p \in T_p P.\qedhere\]
\end{proof}

\subsection{Proper Lie groupoids}

\begin{definition}A Lie groupoid $\G\arrows M$ is called \textbf{proper} if it is Hausdorff and $(s,t):\G\to M\x M$ is a proper map.
\end{definition}

\begin{example}\label{exm-proper-gpd} For several of the examples of Lie groupoids described before the condition of properness becomes some sort of familiar compactness condition.

\begin{itemize}\item[1.] A Lie group $G$ is proper when seen as a Lie groupoid if and only if it is compact.

\item[2.] The submersion groupoid $\G(\pi)$ associated to a submersion $\pi:M\to B$  is always proper.

\item[3.] An action groupoid is proper if and only if it is associated to a proper Lie group action.

\item[4.] The gauge groupoid of a principal $G$-bundle is proper if and only if $G$ is compact.

\item[5.] If $\G\arrows M$ is a proper Lie groupoid and $S\subset M$ a submanifold such that the restriction $\G_S\arrows S$ is a Lie groupoid, then $\G_S$ is proper as well.
\end{itemize}
\end{example}

The following result gives a first glimpse on how proper Lie groupoids are better behaved than general ones.

\begin{proposition}\label{hausdorff}
Let $\G\arrows M$ be a proper Lie groupoid. Then the orbit space $M/\G$ is Hausdorff and the isotropy group $\G_x$ is compact for every $x\in M$. 
\end{proposition}
\begin{proof}Since the map $(s,t):\G\to M\x M$ is proper, it is closed, and has compact fibres. This automatically implies that the isotropy groups are compact and since the orbit space is the quotient of $M$ by the closed relation $(s,t)(\G)\subset M\x M$, it is Hausdorff.
\end{proof}

\begin{proposition} Let $\G$ and $\H$ be Morita equivalent Lie groupoids. If one of them is proper, then the other one is proper as well.
\end{proposition}
\begin{proof} As mentioned in Remark \ref{simplemorita}, in order to prove invariance of a property, we may assume that $\H\tto N$ is equal to the pullback of $\G\tto M$ via a surjective submersion $\alpha:N\to M$. But then we have a pullback diagram relating the maps $(s,t):\G\rmap M\x M$ and $(s',t'):\H\rmap N\x N$. The result follows from stability of proper maps (with Hausdorff domain) under pullback. 
\end{proof}

Before looking at the local structure of a proper Lie groupoid, let us briefly recall the local structure of proper Lie group actions.
Whenever a Lie group $G$ acts on a manifold $M$, we can differentiate the action to get an induced action of $G$ on $TM$, the \textbf{tangent action} of $G$, defined by \[g\cdot X=\dezero (g\cdot x(\epsilon)),\] where $X\in T_xM$ and $x(\epsilon)$ is a curve representing $X$.

For any point $x\in M$, if we restrict this action to an action of the isotropy group $G_x$, then we obtain a representation of $G_x$ on $T_xM$. Since the action of $G_x$ leaves the tangent space to the orbit through $x$ invariant, we obtain an induced representation on the quotient $\N_x=T_xM/T_xO_x$, called the \textbf{isotropy representation} at $x$. This representation is used in the normal form around an orbit for a proper action of a Lie group, described by the Slice theorem, also called Tube theorem \cite[p.\ 109]{DK}.

\begin{theorem}[Slice theorem for proper actions]\label{thm-tube}
Let a Lie group $G$ act properly on a manifold $M$ and let $x\in M$. Then there is an invariant open neighbourhood of $x$ (called a \textit{tube} for the action) which is equivariantly diffeomorphic to $G\x_{G_x}B$, where $B$ is a $G_x$-invariant open neighbourhood of $0$ in $\N_x$ (the isotropy representation).
\end{theorem}

Let us return to the case of a proper Lie groupoid. First, let us recall that there is a pointwise version of properness.

\begin{definition}
Let $\G\arrows M$ be a Lie groupoid and $x\in M$. The groupoid $\G$ is \textbf{proper at $x$} if every sequence $(g_n)\in \G$ such that $(s,t)(g_n)\to (x,x)$ has a converging subsequence.
\end{definition}

\begin{lemma}(\cite{matias}) A Lie groupoid is proper if and only if it is proper at every point and its orbit space is Hausdorff.
\end{lemma}

\begin{definition}\label{def-slice}
Let $\G$ be a Lie groupoid over $M$ and $x\in M$. A \textbf{slice} at $x$ is an embedded submanifold $\Sigma\subset M$ of dimension complementary to $\O_x$ such that it is transverse to every orbit it meets and $\Sigma\cap \O_x=\{x\}$.  
\end{definition}

The following result gives us some information about the longitudinal (along the orbits) and the transverse structure of a groupoid $\G$, at a point $x$ at which $\G$ is proper. For a proof we refer to \cite{ivan}.

\begin{proposition}Let $\G\arrows M$ be a Lie groupoid which is proper at $x\in M$. Then
\begin{enumerate}
\item[1.] The orbit $\O_x$ is an embedded closed submanifold of $M$;
\item[2.] there is a slice $\Sigma$ at $x$.
\end{enumerate}
\end{proposition}

Let $\G\arrows M$ be a Lie groupoid and $\O$ an orbit of $\G$. We recall that the restriction $$\G_\O=\{g\in\G\ |\ s(g),t(g)\in\O\}$$ is a Lie groupoid over $\O$. The normal bundle of $\G_\O$ in $\G$ is naturally a Lie groupoid over the normal bundle of $\O$ in $M$:
$$\N(\G_\O):=T\G/T\G_\O\arrows \N\O:=TM/T\O,$$ with the groupoid structure induced from that of $T\G\arrows TM$. 
The groupoid $\N(\G_\O)$ is called the \textbf{local model}, or the linearization, of $\G$ at $\O$.

Recall also that the restricted groupoid $\G_\O$ has a natural representation on the normal bundle to the orbit, called the normal representation, defined by $g\cdot[v]=[dt(X)]$, for any $v\in T_{s(g)}M$ and any $s$-lift $X\in T_g\G$ of $v$.

This representation can be restricted, for any $x\in M$, to a representation of the isotropy group  $\G_x$ on $\N_x$, also called the \textbf{normal representation}, or isotropy representation, at $x$.

Using this representation, there is a more explicit description of the local model using the isotropy bundle by choosing a point $x$ of $\O$. Let $P_x$ denote the $s$-fibre of $x$ and recall that it is a principal $\G_x$-bundle over $\O_x$ (Proposition \ref{prop-fund-str-gpd}).
Then the normal bundle to $\O$ is isomorphic to the associated vector bundle $$\N_\O\cong P_x\x_{\G_x}\N_x,$$ and the local model is $$\N(\G_\O)\cong (P_x\x P_x)\x_{\G_x}\N_x.$$ The structure on the local model is given by $$s([p,q,v])=[q,v], \ t([p,q,v])=[p,v],\ [p,q,v]\cdot[q,r,v]=[p,r,v].$$

\begin{remark}
Since $\G_\O$ is transitive, it is Morita equivalent to $\G_x$.
Moreover, the linearization $\N(\G_\O)$ is Morita equivalent to $\G_x\lx\N_x$.
\end{remark}

The following linearization result is an essential tool for proving most of the results in this text.

\begin{theorem}[Linearization theorem for proper groupoids]\label{lin}
Let $\G\arrows M$ be a Lie groupoid and let $\O$ be the orbit through $x\in M$. If $\G$ is proper at $x$, then there are neighbourhoods $U$ and $V$ of $\O$ such that $\G_U\cong \N(\G_\O)_V$.
\end{theorem}

The proof of the linearization result around a fixed point (an orbit consisting of a single point) was first completed by Zung \cite{zung}; together with previous results of Weinstein \cite{alan}, it gave rise to a similar result to the one we present here. 

The final version of the Linearization theorem that we present here has appeared in \cite{ivan}, where issues regarding which were the correct neighbourhoods of the orbits to be taken were solved;

\begin{remark} Combining the Linearization theorem with the previous remarks on Morita equivalence, we conclude that any orbit $\O_x$ of a proper groupoid $\G$ has an invariant neighbourhood such that the restriction of $\G$ to it is Morita equivalent to $\G_x\lx \N_x$. For this we use also that $\N_x$ admits arbitrarily small $\G_x$-invariant open neighbourhoods of the origin which are equivariantly diffeomorphic to $\N_x$.
\end{remark}

When we are interested in local properties of a groupoid, it is often enough to have an open around a point in the base, not necessarily containing the whole orbit, and the restriction of the groupoid to it. In this case it is possible to give a simpler model for the restricted groupoid \cite[Cor.\ 3.11]{hessel}:

\begin{proposition}[Local model around a point]\label{prop-local_model_at_point} Let $x\in M$ be a point in the base of a proper groupoid $\G$. There is a neighbourhood $U$ of $x$ in $M$, diffeomorphic to $O\x W$, where $O$ is an open ball in $\O_x$ centred at $x$ and $W$ is a $\G_x$-invariant open ball in $\N_x$ centred at the origin, such that under this diffeomorphism the restricted groupoid $\G_U$ is isomorphic to the product of the pair groupoid $O\x O\arrows O$ with the action groupoid $\G_x \x W\arrows W$.
\end{proposition}

One of the main features of proper groupoids is that it is possible to take averages of several objects (functions, Riemannian metrics, ...) to produce invariant versions of the same objects. This is done using a Haar system, in an analogous way to how one uses a Haar measure on a compact Lie group (cf. \cite{tu}, and also the appendix in \cite{measures_stacks} for further clarification on the constructions).

\begin{definition}\label{Haarsystem}

Given a Lie groupoid $\G$ over $M$, a (right) \textbf{Haar system} $\mu$ is a family of smooth measures $\{\ \mu^x\ |\ x\in M\}$ with each $\mu^x$ supported on the $s$-fibre of $x$, satisfying the properties

\begin{enumerate}
\item[1.] (Smoothness) For any $f\in \CC_c(\G)$, the formula $$I_\mu (f )(x):=\int_{s^{-1}(x)}f(g)\ d\mu^x(g)$$ defines a smooth function $I_\mu(\phi)$ on $M$.

\item[2.] (Right-invariance) For any $h\in \G$ with $h:x\to y$ and any $f\in \CC_c(s^{-1}(x))$ we have $$\int_{s^{-1}(y)}f(gh)\ d\mu^y(g)=\int_{s^{-1}(x)}f(g)\ d\mu^x(g)$$

For such a Haar system, a \textbf{cut-off function} is a smooth function $c$ on $M$ satisfying 

\item[3.] $s: supp(c\circ t)\to M$ is a proper map;
\item[4.] $\int_{s^{-1}(x)}c(t(g))\ d\mu^x(g)=1$ for all $x\in M$.
\end{enumerate}
\end{definition}


\begin{proposition} Any Lie groupoid admits a Haar system and cut-off functions exist for any proper Lie groupoid.
\end{proposition}

\newpage

\subsection{Orbifolds, orbispaces and differentiable stacks}

Lie groupoids can be used in order to conduct differential geometry on singular (i.e. not smooth) spaces. The way to do so is to model the singular space we wish to study as the orbit space of a Lie groupoid, bearing in mind that Morita equivalent Lie groupoids will describe the ``same'' space.

We are interested in spaces which are locally modelled on quotients of Euclidean spaces by smooth actions of \emph{compact} Lie groups (and such that the actions are part of the structure), called \textbf{orbispaces}. 
A particular class of such spaces is that of orbifolds. These are spaces which are locally modelled on quotients of Euclidean spaces by smooth actions of \emph{finite} groups. Orbifolds are more widespread in the literature (see \cite[Ch.\ $8$]{Joyce} for a review of the several definitions of orbifold, and of the category or 2-category of orbifolds in the literature). It is by generalizing their definition in terms of groupoids (cf. \cite{pronk}) that we arrive at the following definition of orbispace.

\begin{definition}\label{dfn-orbispace1}An \textbf{orbispace atlas} on a topological space $X$ is given by a proper Lie  groupoid $\G\arrows M$ and a homeomorphism $f:M/\G\rmap X$.

Two orbispace atlases $(\G,f)$ and $(\H,f')$ are \textbf{equivalent} if $\G\arrows M$ and $\H\arrows N$ are Morita equivalent, and the homeomorphism $\Phi:N/\H\rmap M/\G$ induced by the Morita equivalence (see Lemma \ref{lemma-moritarmk}) satisfies $f\circ \Phi=f'$.

 An \textbf{orbispace} is a topological space equipped with an equivalence class of orbispace atlases. 
Given any proper Lie groupoid $\G\arrows M$, the orbispace associated to it by using the atlas $(\G, id_{M/\G})$ on $M/\G$ is denoted by $M//\G$.
\end{definition}

In this language, an orbifold is simply an orbispace which admits an atlas $(\G,f)$ such that all the isotropy groups of $\G$ are discrete (cf. \cite{pronk,marius_ieke2}).
On the other hand, dropping the condition of properness in the definition of orbispace atlas, we arrive at the notion of differentiable stack. Therefore, orbispaces are examples of differentiable stacks, and orbifolds are examples of orbispaces.



\begin{remark} We warn the reader about the fact that we are avoiding all technicalities related with defining morphisms (and $2$-morphisms) between differentiable stacks (and orbispaces), but we implicitly identify isomorphic orbispaces. Nonetheless, the definitions presented here are sufficient for the scope of this exposition. We refer to \cite{behrend_xu,heinloth,metzler} for comprehensive introductions to the theory of differentiable stacks, and to \cite{matias} for the specific case of orbispaces.
\end{remark}

\section{Orbispaces as differentiable spaces}
\label{chp3-smooth structures}

We study the smooth structure of an orbispace $X$. The approach we follow is to single out the sheaf $\C^\infty_X$ of smooth functions on $X$, and study $(X, \C^\infty_X)$ as a locally ringed space. Throughout this section let $\G\arrows M$ be a proper Lie groupoid with orbit space $X$.

\subsection{Smooth functions on orbit spaces of proper groupoids}

\begin{definition}\label{dfn-functions-on-orbispace} The \textbf{algebra of smooth functions on $X$} is defined as 
\[ C^{\infty}(X):= \{ f: X \rightarrow \mathbb{R}\ |\ f\circ \pi \in C^{\infty}(M)\},\] where $\pi:M\rmap X$ denotes the canonical projection map. 

The \textbf{sheaf of smooth functions on $X$}, denoted by $\C_X^\infty$, is defined by letting \[\C^\infty_X(U):=C^{\infty}(\pi^{-1}(U)/\G_{|\pi^{-1}(U)}).\]
\end{definition}

Note that the pullback map $\pi^*:C^{\infty}(X)\rmap C^{\infty}(M)$ identifies the algebra of smooth functions on $X$ with the algebra of $\G$-invariant smooth functions on $M$, denoted by $\CC(M)^{\G-\mathrm{inv}}$.

The orbit space $X$ of a proper Lie groupoid is Hausdorff, second-countable, and locally compact (since the quotient map $\pi: M \longrightarrow X$ is open for any groupoid $\G\arrows M$ ), hence also paracompact. Using these properties, we are able to guarantee the existence of several useful smooth functions on $X$.

\begin{proposition}\label{CC(X)isnormal} The algebra $\CC(X)$ is \textbf{normal}, i.e., for any disjoint closed subsets $A,B\subset X$ there is a function $f\in\CC(X)$ with values in $[0,1]$ such that $f_{|A}=0$ and $f_{|B}=1$.
\end{proposition}
\begin{proof}Let $A,B\subset X$ be closed and disjoint. The sets $\pi^{-1}(A)$ and $\pi^{-1}(B)$ are closed and disjoint in $M$. Since $M$ is a manifold, we can find a function $h\in \CC(M)$ separating these two sets. Now average $h$ with respect to a Haar system on $\G$. We obtain in this way an invariant smooth function $\tilde{h}$ on $M$ which still separates $\pi^{-1}(A)$ and $\pi^{-1}(B)$. Being invariant, it corresponds via the pullback by the projection $\pi$ to a function $f\in \CC(X)$ that separates $A$ and $B$. 
\end{proof}

The fact that $\CC(X)$ is normal is used to prove the existence of  useful smooth functions on $X$: partitions of unity and proper functions. Before we see how, let us recall the Shrinking Lemma (see for example \cite{dugundji} for a proof).
We say that a cover $\{U_i\}_{i\in I}$ of a space $X$ is locally finite if for each $x\in X$ has a neighbourhood $V_x$ such that there are finitely many indices $i\in I$ for which $V_x\cap U_i\neq \emptyset$.

\begin{lemma}[Shrinking lemma]
Let $X$ be a paracompact Hausdorff space. Then $X$ is normal and for any open cover $\{U_i\}_{i\in I}$ of $X$ there is a locally finite open cover $\{V_i\}_{i\in I}$ with the property that $\{\overline{V_i}\subset U_i\}$ for all $i\in I$.
\end{lemma}

\begin{proposition}[Partitions of unity] For any open cover $\mathcal{U}$ of $X$ there is a smooth partition of unity subordinated to $\mathcal{U}$. 
\end{proposition}
\begin{proof}
One possible proof goes by the standard argument used for the classical version of this result, for continuous functions on a paracompact Hausdorff space \cite{dugundji}. It first involves using the Shrinking lemma twice to obtain locally finite open covers $\{V_n\}$ and $\{W_n\}$ of $X$ such that $\overline{W}_n\subset V_n$ and $\overline{V}_n\subset U_n$ for each $n$. Secondly, since $\CC(X)$ is normal, we can choose functions $f_n:X\to [0,1]$ such that ${f_n}_{|X-V_n}=0$ and ${f_n}_{|\overline{W}_n}=1$. Since the covers used are locally finite, we can define functions $g_n=f_n/\sum f_n$, which form a partition of unity subordinated to $\mathcal{U}$.

There is also a proof using the classical version of the result. Take the open cover $\pi^{-1}(\mathcal{U})$ of $M$ defined by the preimages by the projection map of the opens of $\mathcal{U}$, and consider a smooth partition of unity $(f_n)$ subordinated to it, which exists because $M$ is a manifold. Averaging each function with respect to a Haar system leads to a smooth partition of unity subordinated to $\mathcal{U}$.
\end{proof}

\begin{proposition}[Existence of proper functions]\label{properfunctions}  Let $X$ be the orbit space of a proper groupoid. There exists a smooth proper function $f:X\to \mathbb{R}$.
\end{proposition}
\begin{proof}Let $\{U_n\}$ be a locally finite countable open cover of $X$ such that each $U_n$ has compact closure. Using the Shrinking lemma twice, find open covers $\{V_n\}$ and $\{W_n\}$ of $X$ such that $\overline{W}_n\subset V_n$ and $\overline{V}_n\subset U_n$ for each $n$. Let $f_n\in\CC(X)$ be a function separating  $X-V_n$ and $\overline{W}_n$. This means that $\mathrm{supp}(f_n)\subset U_n$ and $f_n=1$ on a neighbourhood of $\overline{W}_n$. By local finiteness of the covers we can define the function $f=\sum_n nf_n$, which is proper. Indeed, if $K\in \mathbb{R}$ is compact, then it is bounded above by some integer $m$, and so $f^{-1}(K)$ is a closed subspace of the compact $\overline{W}_1,\cup \ldots\cup \overline{W}_m$, and so it is compact itself. \end{proof}







\subsection{Locally ringed spaces}\label{sec-mfd-ringed} Let us recall some background on the more algebro-geometric approach to studying smooth manifolds using the language of locally ringed spaces, as well as some classical results about smooth functions on manifolds that will be later generalized to other spaces. In this section we follow the exposition of \cite{juan}.
Unless otherwise mentioned, all manifolds are considered to be finite dimensional, Hausdorff and second-countable.

\begin{definition}\label{dfn-locally_ringed_space}A \textbf{ringed $\R$-space} (also called ringed space) is a pair $(X,\O_X)$ where $X$ is a topological space and $\O_X$ is a sheaf of $\R$-algebras on  $X$, called the \textbf{structure sheaf} of the ringed space. A \textbf{morphism of ringed spaces} is a pair \[(\phi,\phi^\sharp):(X,\O_X)\to(Y,\O_Y)\] consisting of a continuous map $\phi:X\to Y$ and a morphism  $\phi^\sharp:\O_Y\to\phi_*\O_X$ of sheaves on $Y$.
A \textbf{locally ringed space} is a ringed space $(X,\O_X)$ such that the stalk $\O_{X,x}$ at any point $x\in X$ is a local ring, i.e., it has a unique maximal ideal, denoted by $\m_x$.
A \textbf{morphism of locally ringed spaces} is a morphism of ringed spaces $(\phi,\phi^\sharp):(X,\O_X)\to(Y,\O_Y)$ such that $\phi^\sharp(\m_y)\subset \m_{\phi(y)}$. This condition means that \[(\phi^\sharp f)(x)=0 \Leftrightarrow f(\phi(x))=0.\]
A morphism $(\phi,\phi^\sharp)$ is an isomorphism if $\phi$ is a homeomorphism and $\phi^\sharp$ is an isomorphism of sheaves.

A ringed space $(X,\O_X)$ is said to be \textbf{reduced} if $\O_X$ is a subsheaf of $\R$-algebras of the sheaf $\mathcal{C}_X$ of continuous functions on $X$ and contains all constant functions.

\end{definition}

\begin{remark} In general, the notion of ringed space allows for the structure sheaf to be a sheaf of unital rings, but since all the structure sheaves we use in this text are actually sheaves of $\R$-algebras, we restrict to this class.

Any reduced ringed space is automatically a locally ringed space, the unique maximal ideal of the stalk at a point being the ideal of germs that vanish at that point.
\end{remark}

\begin{example}The following are two essential examples:\begin{enumerate}
\item[1.] Any topological space $X$ is a ringed space if we take $\O_X$ to be equal to $\C_X$, the sheaf of continuous functions.

\item[2.] $(\R^n,\C_{\R^n}^\infty)$ is a locally ringed space, where $\C_{\R^n}^\infty$ is the sheaf of smooth functions on opens of $\R^n$; it is easy to check that any morphism of locally ringed spaces $(\R^n,\C_{\R^n}^\infty)\to (\R^m,\C_{\R^m}^\infty)$ is just given by a smooth map $\R^n\to \R^m$.
\end{enumerate}\end{example}

We can use the language of locally ringed spaces to give an alternative (equivalent) definition of smooth manifolds to the usual one in terms of atlases.

\begin{definition}[Manifolds as locally ringed spaces]\label{dfn-mfd-lrs} A \textbf{smooth manifold of dimension $n$} is a locally ringed space $(M,\O_M)$ such that $M$ is Hausdorff, second-countable, and has a cover by open subsets $U_i$ with the property that each restriction $(U_i,\O_{M|U_i})$ is isomorphic (as a locally ringed space) to an open subset of $(\R^n,\C_{\R^n}^\infty)$.
\end{definition}

We recall the construction of the real spectrum of an $\R$-algebra $A$. To start with, let us see how to construct the underlying set. An ideal $\m$ of $A$ is called a \textbf{real ideal} if it is a maximal ideal and $A/\m\cong \R$. A \textbf{character} on $A$ is a morphism of $\R$-algebras $\chi:A\to \R$; the kernel of any character is a real ideal, so there is a natural bijection  between the set of characters on $A$ and the set of real ideals of $A$.

\begin{definition}
Let $A$ be an $\R$-algebra.  The \textbf{real spectrum} of $A$ is the set $$\Spec_r A:=\Hom(A,\R)=\{\text{real ideals of $A$}\}.$$
\end{definition}

Given a point $x$ in $\Spec_r A$, the corresponding character is denoted by  $\chi_x$ and the corresponding ideal by $\m_x$.

Any element $f\in A$ defines a real valued function $\hat{f}$  (but also denoted by $f$ if there is no risk of confusion) on $\Spec_r A$, given by $$\hat{f}(x):=\chi_x(f)=[f]\in A/\m_x\cong \R.$$

In this way we have that \[\m_x=\{f\in A\ |\ f(x)=0\}\] and  the corresponding character $\chi_x$ is the evaluation map at $x$.

The topology that we consider on the real spectrum $\Spec_r A$ is the \textbf{Gelfand topology}, which is the smallest one such that the functions $\hat{f}: \Spec_r A\to \R$ are continuous, for all $f\in A$.

For any subset $I$ of an $\R$-algebra $A$, we define the \textbf{zero-set} of $I$ to be $$(I)_0:=\{x\in\Spec_r A \ |\ f(x)=0,\ \forall f\in I\}.$$ By definition, these subsets of $\Spec_r A$ are the closed subsets of the \textbf{Zariski topology} on $\Spec_r A$. The Gelfand topology is always finer than the Zariski topology, although they agree in some cases, as we will see below.

Finally, the structure sheaf on $\Spec_r A$ is the sheaf associated to the presheaf that assigns to an open subset $U\subset \Spec_r A$ the ring $A_U$ defined  as the localization (i.e., ring of fractions - see \cite{atiyah}) of $A$ with respect to the multiplicative system of all elements $f\in A$ such that $\hat{f}$ does not vanish at any point of $U$. The stalk at any point coincides with the localization of $A$ with respect to the multiplicative system $\{f\in A\ |\ \hat{f}(x)\neq 0\}$. The resulting locally ringed space $(\Spec_r A,\tilde{A})$ is called \textbf{the real spectrum of $A$} (it has the same name as the underlying set, but the meaning is usually clear from the context).

As mentioned before, a manifold can be recovered from its ring of functions. To start with, the following theorem shows how to recover the underlying topological space.

\begin{theorem}\label{thm-reconstructionM}
Let $M$ be a smooth manifold. Then the map $$\chi: M\to \Spec_r C^\infty(M),$$ given by $\chi(p)(f)=f(p)$ (i.e., $\chi(p)$ is the evaluation at $p$) is a homeomorphism.
\end{theorem}

For a proof see for example \cite{juan} or \cite{nestruev}, but also the proof of Theorem \ref{thm-reconstructionX} below which generalizes this result. From the proof it also follows that for a smooth manifold, the Gelfand and the Zariski topologies coincide.

That the structure sheaf associated to the ring $C^\infty(M)$ by localization coincides with the sheaf $\C_M^\infty$ of smooth functions on opens of $M$ is a consequence of the following result (see e.g. \cite{juan} for a modern exposition of the proof).

\begin{theorem}[Localization theorem]\label{thm-localization-mfd}
Let $M$ be a smooth manifold and $U\subset M$ an open. For any differentiable function $f$ on $U$ there exist global differentiable functions $g,h$ on $M$, such that $h$ does not vanish on $U$ and $f=g/h$ on $U$, i.e., $$\CC(U)=\CC(M)_U.$$
\end{theorem}

Finally, smooth maps between manifolds can also be recovered from algebra maps between the rings of functions.

\begin{theorem}[Theorem 2.3 in \cite{juan}]\label{thm-morphisms-mfds}For any two manifolds $M$ and $N$ there is a natural bijection \[C^\infty(M,N)\to \Hom_{\R-\mathrm{alg}}(C^\infty(N),C^\infty(M))\] given by $\phi\mapsto \phi^*$.
\end{theorem}

With the presented framework in mind, an approach to equipping singular spaces with smooth structures is to choose a class $\mathcal{A}$ of $\R$-algebras, such that the singular spaces we want to consider can be modelled (at least locally) by the real spectrum of algebras in $\mathcal{A}$. Moreover, $\mathcal{A}$ should include the algebra $C^\infty(M)$, for any manifold $M$. There is an ample choice of which kind of algebras should be taken to belong to $\mathcal{A}$ (see for example \cite{Joyce} or Appendix $2$ of \cite{reyes} for a discussion on possible models), depending on which singular spaces we wish to model.

Since our goal is to equip the orbit space $X$ of a proper groupoid $\G$ with a smooth structure, we should ask that the spectrum of $C^\infty(X)$ (see Definition \ref{dfn-functions-on-orbispace}) is locally isomorphic to the spectrum of an algebra in $\mathcal{A}$. We achieve this goal by letting $\mathcal{A}$ consist of differentiable algebras, discussed in the next section. We could also go further and require that the algebra $C^\infty(X)$ itself is in $\mathcal{A}$. We will see how to achieve this in section \ref{subsection-cinfty.rings}. But first let us show that, in any case, we can recover the underlying topological space $X$ from the algebra $C^\infty(X)$. The result generalizes Theorem \ref{thm-reconstructionM} and the proof is exactly the same as the one for manifolds, relying simply on the existence of proper functions on $X$ (Proposition \ref{properfunctions}) and on the fact that $C^\infty(X)$ is normal (Proposition \ref{CC(X)isnormal}).

\begin{theorem}\label{thm-reconstructionX}Let $X$ be the orbit space of a proper Lie groupoid. Then  the natural map $\phi: X \to Spec_r \CC(X)$ given by $x\mapsto ev_x$ is a homeomorphism.
\end{theorem}
\begin{proof} The map $\phi$ is clearly injective since $\CC(X)$ is point-separating.
To check surjectivity, let $\chi \in \Spec_r \CC(X)$. According to Proposition \ref{properfunctions}, we can choose a proper function $f\in \CC(X)$, and then the level set $K = f^{-1}(\chi(f))$ is compact. Suppose that  $\chi$ is not in the image of $\phi$, i.e., it is not given by evaluation at a point. Then for each point $y\in X$ there is a function $f_y\in \CC(X)$ such that $f_y(y)\neq \chi(f_y)$. The sets \[U_y=\{x\in X\ |\ f_y(x)\neq \chi(f_y)\}\] cover $K$, which is compact, so we can take a finite subcover of it, $U_{y_1},\ldots, U_{y_k}$. Consider now the function $$g=(f-\chi(f))^2 + \sum_{i=1}^k (f_{y_i}-\chi(f_{y_i}))^2.$$ It is easy to see that $\chi(g)=0$. But $g$ is a nowhere vanishing smooth function on $X$, so it is invertible and we have that \[1=\chi(1)=\chi\left(g \frac{1}{g}\right)=\chi(g)\chi\left(\frac{1}{g}\right),\] so that $\chi(g)\neq 0$. Thus we have a contradiction and $\phi$ must be surjective. 

The map $\phi$ is continuous, since if $U=\hat{f}^{-1}(V)$ is some basic open for the Gelfand topology, with $f\in \CC(X)$ and $V$ open in $\mathbb{R}$ , then $\phi^{-1}(U)=f^{-1}(V)$ is open in $X$. 

To finish checking that $\phi$ is a homeomorphism, let $\phi$ induce a topology on $X$ which we also call the Gelfand topology. Given a set $Y\subset X$ which is closed in the original topology of $X$, we can consider the ideal $I_Y$ consisting of all functions of $\CC(X)$ vanishing on $Y$. Since the algebra $\CC(X)$ is normal, we have that $Y=\{x\in X\ |\ f(x)=0\ \ \forall\ f\in I_Y\}$, so $Y$ is also closed in the Gelfand topology.
\end{proof}

\subsection{Differentiable spaces}

In this section we discuss the category of differentiable spaces, which will turn out to be a good setting for the study of orbispaces. These are spaces that are locally modelled on the spectrum of differentiable algebras (also called closed $C^\infty$-rings, cf. \cite{reyes}), i.e., algebras of the form $C^\infty(\R^n)/I$, where $I$ is an ideal of $C^\infty(\R^n)$ which is closed with respect to the weak Whitney topology (cf. e.g. \cite{hirsch}). Differentiable spaces appeared in the work of Spallek \cite{spallek} and also as a particular case of the theory of $C^\infty$-schemes \cite{dubuc,Joyce,reyes}, that is discussed in the next section. A study of differentiable spaces, analogous to the basics of scheme theory in algebraic geometry, is discussed in detail in the book \cite{juan}; this is the main reference for this section. See \textit{loc. cit.} for the proofs of the results quoted below.


We deal with ideals of $C^\infty(M)$ that are closed in the Fr\'echet topology on $C^\infty(M)$ (also called the weak Whitney topology \cite{hirsch}). An example of such a closed ideal is the ideal $\mathfrak{m}_p$ of all smooth functions vanishing at a point $p\in M$.
The following classical result of Whitney characterizes the closure of an ideal of $C^\infty(M)$ in terms of conditions on the jets of smooth functions (cf. e.g. \cite{malgrange} for a proof).

\begin{theorem}[Whitney's spectral theorem] Let $M$ be a smooth manifold and $\a$ an  ideal of $\CC(M)$. Then \[f\in \bar{\a}\ \Leftrightarrow \ j_x f\in j_x(\a), \forall x\in M,\] where $j_xf$ denotes the jet of $f$ at the point $x$ and $j_x(\a)=\{j_xg\ |\ g\in \a\}$ .
\end{theorem}

\begin{definition}\label{dfn-diff-alg}
An $\R$-algebra $A$ is called a \textbf{differentiable algebra} if it is isomorphic to the quotient $C^\infty(\R^n)/\a$, where $\a$ is a closed ideal (with respect to the Fr\'echet topology). Morphisms of differentiable algebras are simply morphisms of the underlying $\R$-algebras.
\end{definition}

\begin{example} \ \
\begin{itemize} 
\item[1.] (Open or closed subsets of Euclidean space) If $U\subset\R^n$ is an open subset, then $C^\infty(U)$ is a differentiable algebra. 

If $Y\subset \R^n$ is a closed subset, denote by $\mathfrak{p}_Y$ the ideal of all smooth functions on $\R^n$ which vanish on $Y$. Define
\[A_Y:=\{f_{|Y}\ |\ f\in \C^\infty(\R^n)\}.\]
Then $A_Y$ is a differentiable algebra because $A_Y\cong C^\infty(\R^n)/\mathfrak{p}_Y$ via the map $[f]\mapsto f_{|Y}$, and the ideal $\mathfrak{p}_Y$ is closed, since $\mathfrak{p}_Y=\bigcap_{p\in Y}\mathfrak{m}_p$.
\item[2.] (Smooth manifolds) As a particular case of the previous example, using Whitney's embedding theorem (see e.g. \cite{hirsch}) we conclude that the algebra of smooth functions on any manifold is a differentiable algebra.
\end{itemize}
\end{example}

We now study the real spectrum of a differentiable algebra (see the discussion in Section \ref{sec-mfd-ringed}). The point is that for a differentiable algebra $A$, the real spectrum behaves similarly to how a manifold does. For example, closed (resp. open) subsets of $\Spec_r A$ share many features with closed (resp. open) subsets of smooth manifolds; this is made more precise by the next two propositions.

\begin{proposition}[Closed subsets - Lemma 3.1 in \cite{juan}] If $A$ is a differentiable algebra and $Y\subset \Spec_r A$ is a closed subset, then $Y$ is a zero-set, i.e., there is an element $a\in A$ such that $Y=(a)_0$.
\end{proposition}

Note that the Gelfand and the Zariski topologies on $\Spec_r A$ coincide for any differentiable algebra $A$, since any closed subset of $\Spec_r A$ is a zero-set.
For open subsets, the following results extend Theorems \ref{thm-reconstructionM} and \ref{thm-localization-mfd} to the setting of differentiable algebras.

\begin{proposition}[Open subsets - Proposition 3.2 in \cite{juan}] If $A$ is a differentiable algebra and $U\subset \Spec_r A$ is an open subset, then we have a homeomorphism \[U\cong \Spec_r A_U,\] where $A_U$ denotes the localization of $A$ with respect to the multiplicative system of elements of $A$ that vanish nowhere in $U$.
\end{proposition}

It is also important to note that for a differentiable algebra $A$ and an open subset $U\subset \Spec_r A$, the localization $A_U$ is again a differentiable algebra \cite[Thm.\ 3.7]{juan}.

Finally, the following result is essential when considering the spectrum of a differentiable algebra as a locally ringed space.

\begin{theorem}[Localization theorem for differentiable algebras] Let $A$ be a differentiable algebra and let $(\Spec_r A, \tilde{A})$ be its real spectrum. Then for any open subset $U\subset\Spec_r A$ we have that \[\tilde{A}(U)=A_U.\]
\end{theorem}

So we see that the essential properties that allow to reconstruct a manifold from its ring of smooth functions as the real spectrum still hold in the more general setting of differentiable algebras.

\begin{definition}\label{dfn-differentiable_space}
An \textbf{affine differentiable space} is a locally ringed space $(X,\O_X)$ which is isomorphic to the real spectrum $(\Spec_r A, \tilde{A})$ of some differentiable algebra $A$. By the Localization theorem, $A$ must be isomorphic to $\O_X(X)$.

A \textbf{differentiable space} is a locally ringed space $(X,\O_X)$ for which every point $x$ of $X$ has a neighbourhood $U$ such that $(U,\O_{X|U})$ is an affine differentiable space.
Such opens are called \textbf{affine opens}.

\textbf {Morphisms of differentiable spaces} between $(X,\O_X)$ and $(Y,\O_Y)$ are defined to be the morphisms of locally ringed spaces between them (Definition \ref{dfn-locally_ringed_space}). Sections of the sheaf $\O_X$ over an open subset $U\subset X$ are called \textbf{differentiable functions} on $U$.
\end{definition}

\begin{example}\label{exm-cloded-reduced} \ \ 
\begin{itemize}
\item[1.] (Manifolds) Any smooth manifold $M$ is an example of an affine differentiable space. As discussed before, given a manifold $M$, its algebra of smooth functions $C^\infty(M)$ is a differentiable algebra; the manifold $(M,\C_M^\infty)$ is isomorphic, as a locally ringed space, to the real spectrum of $C^\infty(M)$.
\item[2.] (Open subsets) If $(X,\O_X)$ is an affine differentiable space and $U$ is an open subset of $X$ then $(U,\O_{X|U})$ is an affine differentiable space.
\item[3.] (Closed subsets) Let $(X,\O_X)$ be a differentiable space and $Y\subset X$ a closed subset. If $\mathcal{I}_Y$ is the sheaf of differentiable functions vanishing on $Y$, then $(Y,\O_X/\mathcal{I}_Y)$ is a differentiable space.
\end{itemize}
\end{example}

Affine differentiable spaces can be explicitly described, at least as a topological space, as follows.

\begin{lemma}
Let $I$ be an ideal of an $\R$-algebra $A$. Then there is a natural homeomorphism $$Spec_r(A/I)\cong(I)_0\subset \Spec_r A.$$
\end{lemma}

\begin{proposition}[Structure of affine spaces - Proposition 2.13 in \cite{juan}]\label{prop-affine-d-space} Let $A$ be an algebra of the form $A=\CR{n}/\a$, for any ideal $\a\subset \CR{n}$. Then there is a homeomorphism \[\Spec_r A=(\a)_0\subset \R^n.\]
\end{proposition}

As mentioned in the previous section, morphisms of differentiable spaces $M\to N$ between smooth manifolds are simply smooth maps, and these correspond to algebra maps  $C^\infty(N)\to C^\infty(M)$. A similar result is valid for the more general setting of affine differentiable spaces.

\begin{theorem}[Morphisms into affine spaces - Theorem 3.18 in \cite{juan}] If $(X,\O_X)$ is a differentiable space and $(Y,\O_Y)$ is an affine differentiable space, then \[\Hom(X,Y)\cong \Hom_{\R-\mathrm{alg}}(\O_Y(Y),\O_X(X)),\ \ (\phi,\phi^\sharp)\mapsto \phi^\sharp.\]
\end{theorem}

As a particular case of this result, we obtain a characterization of morphisms from a differentiable space to an Euclidean space.

\begin{corollary}[Morphisms into Euclidean space] If $(X,\O_X)$ is a differentiable space, then we have an isomorphism \[\Hom(X,\R^n)\cong \bigoplus_{i=1}^n \O_X(X),\ \ (\phi,\phi^\sharp)\mapsto (\phi^\sharp(x_1),\ldots,\phi^\sharp(x_n)). \]
\end{corollary}

We have seen in the general discussion about the real spectrum of an algebra $A$ that any element $f\in A$ can be seen as a continuous function $\hat{f}$ on $\Spec_r A$. We now focus on the case in which the assignment $f\mapsto \hat{f}$ is injective.

\begin{definition} A differentiable space $(X,\O_X)$ is said to be \textbf{reduced} if for any open subset $U$ of $X$ and any $f\in\O_X(U)$ we have \[f=0\Leftrightarrow \hat{f}(x)=0,\ \forall x\in X.\]
\end{definition}

Being reduced is a local condition: If every point of $x$ has a reduced open neighbourhood then $X$ is reduced.
By definition, if $(X,\O_X)$ is a reduced differentiable space, then the map $\O_X(U)\to C(U)$, $f\mapsto\hat{f}$ is injective for any open $U$; hence $(X,\O_X)$ is a reduced ringed space. 

Let $(\phi,\phi^\sharp)$ be a morphism of differentiable spaces (Definition \ref{dfn-differentiable_space}) between reduced differentiable spaces. Then it can be checked (cf. \cite{juan}) that $(\phi,\phi^\sharp)$ is a morphism of reduced ringed spaces (Definition \ref{dfn-locally_ringed_space}).

\begin{example} Smooth manifolds are reduced differentiable spaces. Differentiable spaces of the form $(Y,\O_X/\mathcal{I}_Y)$, where $Y$ is a closed subset of a differentiable space $X$, as in Example \ref{exm-cloded-reduced} are also reduced.
\end{example}

We now look at an explicit description of affine reduced differentiable spaces. 

\begin{definition} Let $Y$ be a topological subspace of $\R^n$. A \textbf{smooth function on $Y$} is a continuous function $f:Y\to \R$ for which every point $y\in Y$ has an open neighbourhood $U_y$ in $\R^n$ such that $f$ coincides on $Y\cap U_y$ with the restriction of a smooth function on $U_y$. The smooth functions on $Y$ form the ring $C^\infty(Y)$. Denote by $\C^\infty_Y$ the sheaf of continuous functions on $Y$ defined by $\C^\infty_Y(V):=\CC(V)$, for each open subset $V\subset Y$.
\end{definition}

\begin{proposition}[Structure of reduced affine spaces - \cite{juan}, Prop.\ 3.22]\label{prop-reduced-affine}
Let $A=\CR{n}/\a$ be a differentiable algebra. If the affine differentiable space $Y=\Spec_r A$ is reduced, then $\O_Y=\C^\infty_Y$.
\end{proposition}

\begin{remark} We knew already from Proposition \ref{prop-affine-d-space} that in the conditions of the previous proposition, $Y\cong(\mathfrak{a})_0$ as a topological space, because $A=\CR{n}/\a$ is affine. The new information we get from knowing that $Y$ is reduced is the characterization of the structure sheaf. In conclusion, reduced differentiable spaces are reduced ringed spaces which are locally isomorphic to $(Z,\C^\infty_Z)$, for some closed subset $Z$ of $\R^n$. 
\end{remark}

We now discuss the Embedding theorem for differentiable spaces, which characterizes which differentiable spaces can be embedded into an affine space $\R^n$. Before stating the theorem we discuss subspaces and embeddings of differentiable spaces.

\begin{definition}
Let $(X,\O_X)$ be a differentiable space and let $Y\subset X$ be a locally closed subspace. Let $\I$ be a sheaf of ideals of $\O_{X|Y}$ and set $\O_X/\I:=(\O_{X|Y})/\I$.

We say that $(Y,\O_X/\I)$ is a \textbf{differentiable subspace} of $(\O_X,X)$ if it is a differentiable space.
It is said to be an \textbf{open} differentiable subspace if $Y$ is open in $X$ and $\I=0$. It is said to be a \textbf{closed} differentiable subspace if $Y$ is closed in $X$.
\end{definition}

\begin{definition} 
An \textbf{embedding} of differentiable spaces is a morphism of differentiable spaces $(\phi,\phi^\sharp):(Y,\O_Y)\rmap (X,\O_X)$ such that \begin{enumerate}\item[1.] $\phi:Y\rmap X$ induces a homeomorphism of $Y$ onto a locally closed subspace of $X$ \item[2.]$\phi^\sharp:\phi^*\O_X\rmap \O_Y$ is surjective.\end{enumerate} It is called a \textbf{closed embedding} if additionally $\phi(Y)$ is closed in $X$.
\end{definition}




\begin{definition} Let $p$ be a point of a differentiable space $(X,\O_X)$ and let $\m_p$ be the unique maximal ideal of $\O_{X,p}$. The \textbf{tangent space of $X$ at $p$} is defined as: \[T_pX=Der(\O_{X,p}\, ;\, \O_{X,p}/\m_p).\] The dimension of $T_pX$ is called the \textbf{embedding dimension} of $X$ at $p$.
\end{definition}

\begin{theorem}[Embedding theorem for differentiable spaces - cf. \cite{juan}]\label{embedding} A differentiable space is affine if and only if it is Hausdorff, second-countable, and has bounded embedding dimension.
\end{theorem}

\subsection{Orbispaces as differentiable spaces}

We look at how orbit spaces of proper groupoids can be seen as differentiable spaces. We start by discussing the smooth structure on orbit spaces of representations of compact Lie groups.
Let $G$ be a compact Lie group and let $V$ be a representation of $G$.

\begin{definition}\label{dfn-functions-on-orbit-space-of-V} The \textbf{algebra of smooth functions on $V/G$} is defined as 
\[ C^{\infty}(V/G):= \{ f: V/G \rightarrow \mathbb{R} \ |\ f\circ \pi \in C^{\infty}(V)\},\] where $\pi:V\rmap V/G$ denotes the canonical projection map. 

The \textbf{sheaf of smooth functions on $V/G$}, denoted by $\C_{V/G}^\infty$, is defined by letting \[\C^\infty_{V/G}(U):=C^{\infty}(\pi^{-1}(U)/G).\]
\end{definition}

It is natural to identify the algebra of smooth functions on the orbit space, $C^\infty(V/G)$, with the algebra of $G$-invariant smooth functions on $V$, via the pullback map $\pi^*$. 

We now explain how $V/G$ can be seen as an affine differentiable space.
The first step in this direction is given by the following classical result of Schwarz \cite{Schwarz}.

\begin{theorem}[Schwarz]\label{thm-schwarz} Let $G$ be a compact Lie group and $V$ a representation of $G$. Let $p_1,\ldots,p_k$ be generators of the algebra of invariant polynomials $\mathbb{R}[V]^G$. Then $p:V\to \R^k$ defined by $p=(p_1,\ldots,p_k)$ induces an isomorphism \[p^*C^\infty(\R^k)\cong C^\infty(V)^G.\]
\end{theorem}

This allows us to make sense of $V/G$ as a differentiable space. To start with, the map $p$ from the theorem is constant along orbits; so it induces a map  on $V/G$, denoted by $\tilde{p}:V/G\to \mathbb{R}^k$.

\begin{lemma}With the notation from Schwarz's theorem, the map $p:V\to \R^k$ is proper and it induces a closed embedding (of topological spaces) \[\tilde{p}:V/G\to \R^k.\]
\end{lemma}

\begin{remark}In fact, since $p$ is a polynomial map, the image of $V/G$ is naturally a semialgebraic set. Its semialgebraic structure can be described explicitly \cite{Schwarz2}. 
\end{remark}

We can rephrase Schwarz's theorem in the language of differentiable spaces as follows.

\begin{theorem}\label{thm-schwarz-diff} Let $G$ be a compact Lie group and $V$ a representation of $G$. Then \begin{enumerate}\item[1.]{ $(V/G, \C^\infty_{V/G})$ is a reduced affine differentiable space;}
\item[2.]{the map $\tilde{p}:V/G\to \R^k$ is a closed embedding of differentiable spaces.}\end{enumerate}
\end{theorem}





For a proper Lie groupoid $\G\arrows M$, we prove that the orbit space $X$ is a reduced differentiable space. We also see that the differentiable space structure on $X$ only depends on the Morita equivalence class of $\G$, so that it is really associated to the orbispace presented by $\G$.

We know from the Linearization theorem (Theorem \ref{lin}) that any point in $X$ has a neighbourhood homeomorphic to a space of the form $V/G$, where $V$ is a representation of a compact Lie group $G$. The idea is to upgrade this homeomorphism to an isomorphism of locally ringed spaces and to use the differentiable space structure on $V/G$ described in the previous section.

\begin{theorem}[Main Theorem 1]\label{Xissmooth} The orbit space $X$ of a proper Lie groupoid $\G\arrows M$, together with the sheaf $\C^\infty_X$ on $X$ (Definition \ref{dfn-functions-on-orbispace}), is a reduced differentiable space. Moreover, it is affine if and only if it has bounded embedding dimension.

If two Lie groupoids $\G$ and $\H$ are Morita equivalent, their orbit spaces are isomorphic as differentiable spaces.
\end{theorem}

\begin{proof} Let $\O\in X$ and let $x$ be a point in the orbit $\O$. Consider an open subset $U\subset M$ containing $\O$, such that $\G_U\cong \N_{\O}(\G)_V$, as in the Linearization theorem (Theorem \ref{lin}).
Under this isomorphism, let $\Sigma\subset U$ be the slice (see def. \ref{def-slice}) at $x$ corresponding to the open $V\cap \N_x$ of the normal space to $\O$ at $x$. 
It also holds that \[\CC(U)^\G\cong \CC(V)^{\N_\O(\G)}.\] Moreover, since we are only considering invariant functions, we could work with invariant opens instead - if $\tilde{V}$ is the saturation of $V$, we have that \[\CC(\tilde{V})^{\N_\O(\G)}\cong \CC(V)^{\N_\O(\G)}.\] So from now on assume that $\tilde{V}=P_x\x_{\G_x}W$, where $P_x$ is the $s$-fibre of $\G$ at $x$ and $W\subset \N_x$ is an invariant open subset.
Note that $W$ intersects all the orbits in $V$ (it is actually a slice at $x$). We can define  an algebra isomorphism $$\phi: \CC(\tilde{V})^{\N_\O(\G)}\to \CC(W)^{\G_x}$$ by restriction to the slice $W$. More precisely, given a function $f\in \CC(\tilde{V})^{\N_\O(\G)}$ define $\phi(f)(w)=f([1_x,w])$. The invariance of $\phi(f)$ follows from that of $f$ and it is also easy to check that $\phi$ is an injective algebra homomorphism. An explicit inverse can be given by $\phi^{-1}(f)([p,w])=f(w)$. Once again it is easy to check that it is well defined, and an inverse to $\phi$. To see that it is invariant we use that a point of $V$ belonging to the orbit of $[p,w]$ must be of the form $[q,w]$. To summarize, we have isomorphisms \[C^\infty(\pi(U))\cong C^\infty(U)^\G\cong C^\infty(V)^{\N_\O(\G)}\cong C^\infty(W/\G_x).\]

Since $\G_U$ is Morita equivalent to $\G_x\lx W$, it holds, as explained in Lemma \ref{lemma-moritarmk}, that $\pi(U)\cong W/\G_x$; what we have proved above is that in fact, we also obtain an isomorphism of reduced ringed spaces 
\[\big(\pi(U),\C^\infty_{\pi(U)}\big)\cong \big(W/\G_x, \C^\infty(W)^{\G_x}\big).\]
From the discussion of the previous section, we know that Schwarz's theorem implies that $\big(W/\G_x, \C^\infty(W)^{\G_x}\big)$ is a reduced affine differentiable space (Theorem \ref{thm-schwarz-diff}). We have thus proved that the reduced ringed space $(X,\C^\infty_X)$ is locally isomorphic to a reduced affine differentiable space, so it is a reduced differentiable space.

The fact that $X$ is affine if and only if it has bounded embedding dimension follows from  the Embedding theorem for differentiable spaces (Theorem \ref{embedding}).

We already knew that if $\G$ and $\H$ are Morita equivalent, then the orbit spaces of $\G$ and $\H$ are homeomorphic (Lemma \ref{lemma-moritarmk}). We have also seen that as a differentiable space, the orbit space of a proper groupoid is locally isomorphic, around a point $\O$, to the orbit space of the isotropy representation at a point of the orbit $\O$, which is invariant under Morita equivalences.
\end{proof}

\begin{remark}A direct consequence of the last statement is that the algebra of smooth functions on the orbit space $X$ of a proper groupoid is a differentiable algebra if and only if $X$ has bounded embedding dimension.

A case in which $X$ has bounded embedding dimension is for example when $X$ has only finitely many Morita types, a notion that is discussed in Section \ref{sec-Morita_types}.
\end{remark}

\subsection{Alternative framework I: $C^\infty$-Schemes}\label{subsection-cinfty.rings}

In this section we briefly discuss $C^\infty$-schemes. These are spaces locally modelled on the spectrum of a $C^\infty$-ring. They have appeared as models for synthetic differential geometry in \cite{dubuc,reyes}. We have already seen examples of such spaces, as any differentiable algebra is also a $C^\infty$-ring (and so any differentiable space is a $C^\infty$-scheme). Some references for this material are \cite{Joyce,reyes}.


\begin{definition}\label{dfn-Cinfty-ring}
A \textbf{$\CC$-ring} is a set $\CCC$ together with operations $\phi_f:\CCC^n\to \CCC$ for each $n\geq 0$ (also denoted by  $\phi_f^\CCC$) and for each smooth map $f:\R^n\to \R$, satisfying the following conditions.
\begin{enumerate}
\item[1.] Consider natural numbers $m,n \geq 0$, and smooth functions $f_i:\R^n\to\R$ for $i=1,\ldots, m$, $g:\R^m\to \R$. Define a smooth function $h:\R^n\rmap \R$ by \[h(x_1,\ldots,x_n)=g(f_1(x_1,\ldots,x_n),\ldots,f_m(x_1,\ldots,x_n)).\] Then we have that the following diagram commutes:
\[
\begin{tikzcd}[column sep=huge]\CCC^n  \arrow{rd}[swap]{\phi_h}\arrow{r}{(\phi_{f_1},\ldots,\phi_{f_m})} & \CCC^m \arrow{d}{\phi_g} \\
 & \CCC
\end{tikzcd}
\]
\item[2.] For the coordinate functions $x_i:  \R^n \to \R$ we have $\phi_{x_i}(c_1,\ldots,c_n)=c_i$, for all $c_1,\ldots,c_n\in \CCC$.
\end{enumerate}

A morphism of $C^\infty$-rings is a map $F:\CCC\to \mathfrak{D}$ such that for any $f\in C^\infty(\R^n)$ and any $c_1,\ldots, c_n\in \CCC$ it holds that \[F(\phi^\CCC_f(c_1,\ldots,c_n))=\phi_f^{\mathfrak{D}}(F(c_1),\ldots,F(c_n)).\]
\end{definition}

\begin{proposition}[Proposition 1.2 in \cite{reyes}]\label{prop-quotientCinfty}Let $\CCC$ be a $\CC$-ring and $I$ an ideal of $\CCC$ (with $\CCC$ considered as an $\R$-algebra). Then there is a unique $C^\infty$-ring structure on the quotient $\CCC/I$, such that the quotient map $\pi:\CCC\to \CCC/I$ is a $\CC$-ring morphism.
\end{proposition}

As an immediate consequence of this proposition we obtain the following.

\begin{corollary}\label{cor-diff-is-Cinfty} Any differentiable algebra is a $C^\infty$-ring.
\end{corollary}

It is worth mentioning that $C^\infty$-rings are much more general than differentiable algebras - for example, from the theory of the previous section, we know that the algebra of smooth functions on the orbit space of a proper groupoid is a differentiable algebra if and only if the orbit space has bounded embedding dimension. But we will see that it is always a $C^\infty$-ring. As a more extreme example, the algebra of continuous functions on any topological space is a $C^\infty$-ring. 

On the other hand, although $C^\infty$-rings are very general, those that arise as algebras of smooth functions on $X$ still belong to a somewhat restrictive class - that of locally fair $C^\infty$-rings.

\begin{definition}A $\CC$-ring $\CCC$ is said to be  \textbf{finitely generated} if there are $c_1,\ldots,c_n\in \CCC$ which generate $\CCC$ under all $C^\infty$ operations.
  

A $\CC$-ring $\CCC$ is called a \textbf{$\CC$-local ring} if it has a unique maximal ideal $\m$ and $\CCC/\m\cong \R$.
\end{definition}

\begin{example} The ring $\CC_p(\R^n)$ of germs at $p$ of functions on $\R^n$ is a $\CC$-local ring.
\end{example}




\begin{definition}
A \textbf{$\CC$-Ringed space} $(X,\O_X)$ is a topological $X$ space together with a sheaf $\O_X$ of $\CC$-rings on it. 
A \textbf{local $\CC$-Ringed space} $(X,\O_X)$ is a $\CC$-ringed space for which the stalks $\O_{X,x}$ are local rings for all $x\in X$.

A \textbf{morphism of $\CC$-Ringed spaces} is a pair \[(\phi,\phi^\sharp):(X,\O_X)\to(Y,\O_Y)\] consisting of a continuous map $\phi:X\to Y$ and a morphism  $\phi^\sharp:\O_Y\to\phi_*\O_X$ of sheaves  on $Y$ (or equivalently, a morphism $\phi^\sharp:\phi^*\O_Y\to\O_X$ of sheaves of $C^\infty$-rings on $X$).
\end{definition}

\begin{definition}\label{dfn-Cinfty-scheme} An \textbf{affine $\CC$-scheme} is a local $\CC$-ringed space $(X,\O_X)$ which is isomorphic to $\Spec_r\CCC$ as a local $\CC$-ringed space, for  some $\CC$-ring $\CCC$. 

 A \textbf{$\CC$-scheme} is a local $\CC$-ringed space $(X,\O_X)$ for which $X$ can be covered by open sets $U_i$ such that each $(U_i, \O_{X|U_i})$ is an affine $\CC$-scheme. \textbf{Morphisms of $\CC$-schemes} are just morphism of $\CC$-Ringed spaces.
\end{definition}

\begin{definition} A \textbf{locally fair} $C^\infty$-scheme  is a  $C^\infty$-scheme  $(X,\O_X)$ for which $X$ can be covered by open sets $U_i$ such that each $(U_i, \O_{X|U_i})$ is isomorphic to $\Spec_r \CCC_i$, where $\CCC_i$ is a finitely generated $\CC$-ring.
\end{definition}


We return to the smooth structure on the orbit space $X$ of a proper groupoid. Although we already knew that $X$ is a $C^\infty$-scheme (since it is a differentiable space) we show that it is always affine, as a $C^\infty$-scheme.

\begin{proposition}[First variation of Main Theorem 1]\label{Prop-Var1} Let $X$ be the orbit space of a proper Lie groupoid $\G\arrows M$. Then the algebra $C^\infty(X)$ is a $C^\infty$-ring and $(X,\C^\infty_X)$ is a locally fair affine  $C^\infty$-scheme.
\end{proposition}
\begin{proof} It is clear that $C^\infty(X)$ is a $C^\infty$-ring : the operation $\phi_f$ associated with a smooth function $f:\R^n\to \R$ is simply given by composition with $f$.

We also know from Theorem \ref{thm-reconstructionX} that the evaluation map \[ev:X\cong \Spec_r C^\infty(X)\] is a homeomorphism. In fact, almost by definition, $ev$ is also an isomorphism of reduced ringed spaces: Let $f\in \O_{\Spec_r C^\infty(X)}(U)$ and recall that $\O_{\Spec_r C^\infty(X)}(U)$ is the localization of $C^\infty(X)$ with respect to the multiplicative system of those functions of $C^\infty(X)$ that vanish nowhere on $U$. We then know that $f=\frac{g}{h}$ with $g,h\in C^\infty(X)$ and $h$ non-vanishing on $U$, so it is easy to see that $ev^*f\in C^\infty_X(ev^{-1}(U))$.
This implies that $(X,\C^\infty_X)$ is an affine $C^\infty$-scheme.

In the proof of Theorem \ref{Xissmooth}, we have seen that $X$ is locally isomorphic to differentiable spaces (hence $C^\infty$-schemes) of the form $(V/G,\C^\infty_{V/G})$, where $V$ is a representation of a compact Lie group $G$. Schwarz's theorem (Theorem \ref{thm-schwarz}) says that the algebra of $G$-invariant functions on $V$ is a finitely generated $C^\infty$-ring, and that $(V/G,\C^\infty_{V/G})$ is affine; hence $(V/G,\C^\infty_{V/G})$ is isomorphic to the real spectrum of a finitely generated $C^\infty$-ring. Therefore $(X,\C^\infty_X)$ is a locally fair $C^\infty$-scheme.
\end{proof}



\subsection{Alternative framework II: Sikorski spaces}\label{sec-sikorski}

In this section we discuss another notion of smooth structure - that of a Sikorski space \cite{sikorski} (also called differential space in the literature). This type of smooth structure will be revisited later on, when we discuss differentiable stratifications of the orbit space (Section \ref{chp5-smooth_stratifications}). The main references used for this section are \cite{Sniat-book,Watts}.

\begin{definition}
A \textbf{Sikorski space}\label{def-sikorski} is a pair $(X,\F)$, where $X$ is a topological space and $\F$ is a non-empty set of continuous real-valued functions on $X$, satisfying:
\begin{enumerate}
\item[1.] $X$ has the weakest topology such that all the elements of $\F$ are continuous.
\item[2.] (Locality) Let $f:X\to\R$ be a function such that for all $x\in X$ there is a neighbourhood $U$ of $x$ and a function $g\in\F$ such that $f_{|U}=g_{|U}$. Then $f\in\F$.
\item[3.] (Smooth compatibility) If $F\in\CR{n}$ and $f_1,\ldots,f_n\in \F$, then the composition $F(f_1,\ldots,f_n)$ belongs to $\F$.
\end{enumerate}
 The elements of $\F$ are called \textbf{smooth functions} on $X$.
\begin{remark}\label{rmk-sikorski} There are some immediate observations we can draw from the previous definition.

\begin{enumerate}

\item[1.] Since the composition of elements of $\F$ with translations and rescalings is again in $\F$, the topology of $X$ is generated by the open subsets of the form $f^{-1}(0,1)$, for $f\in\F$.

\item[2.] The smooth compatibility condition ensures that $\F$ is a commutative $\R$-algebra and that it contains all constant functions.

\item[3.] The locality condition guarantees that $\F$ induces a sheaf of continuous functions $\tilde{\F}$ on $X$: for any open $U\subset X$, let $\tilde{\F}(U)$ be the set of all functions $f:U\to \R$ such that for all $x\in X$ there is a neighbourhood $V$ of $x$ in $U$ and a function $g\in\F$ such that $f_{|V}=g_{|V}$. 
In this way $(X,\tilde{\F})$ is a reduced ringed space.
\end{enumerate}
\end{remark}

\end{definition}

\begin{definition}
A \textbf{smooth map} between Sikorski spaces $(X,\F_X)$ and $(Y,\F_Y)$ is any continuous map $\phi:X\to Y$ with the property that $f\circ \phi\in \F_X$ for all $f\in\F_Y$.A \textbf{diffeomorphism} is a smooth homeomorphism with a smooth inverse.
\end{definition}

\begin{definition}
Let $(X,\F_X)$ be a Sikorski space. A \textbf{Sikorski subspace} of $(X,\F_X)$ is a Sikorski space $(Y,\F_Y)$, where $Y$ is a topological subspace of $X$, and $\F_Y$ is generated by restrictions of functions of $\F_X$ to $Y$, i.e., $f\in \F_Y$ if and only if for all $y\in Y$ there is a neighbourhood $U$ of $y$ in $X$ and a function $g\in\F$ such that $f_{|U\cap Y}=g_{|U\cap Y}$.
\end{definition}

\begin{definition}
Let $(X,\F_X)$ be a Sikorski space and let $R$ be an equivalence relation on $X$. The \textbf{quotient Sikorski space} of $X$ with respect to $R$ is the Sikorski space $(X_R,\F_R)$, where\begin{enumerate} \item[1.] $X_R$ is the set of equivalence classes for $R$; \item[2.] $\F_R=\{f\in \F \ |\ \pi^*f\in \F\}$, where $\pi:X\to X_R$ is the canonical projection; \item[3.] the topology on $X_R$ is the smallest one making all functions in $\F_R$ continuous.\end{enumerate}
\end{definition}

The topology induced on $X_R$ by $\F_R$ may be distinct from the quotient topology. The following result provides a sufficient condition for the two topologies to coincide.

\begin{proposition}[Proposition 2.1.11 in \cite{Sniat-book}]\label{prop-toplogiescoincide}
Let $(X,\F_X)$ be a Sikorski space, let $R$ be an equivalence relation on $X$ and let $(X_R,\F_R)$ be the quotient Sikorski space. Then the topology induced by $\F_R$ on $X_R$ coincides with the quotient topology if for every subset $U\subset X_R$, open for the quotient topology, and every point $y\in U$, there exists a function $f\in \F_R$ such that $f(y)=1$ and $f$ vanishes outside of $U$.
\end{proposition}

As in the other frameworks presented previously, we will see that orbispaces will fall into a particularly nice subcategory of Sikorski spaces.

\begin{definition}\label{dfn-subcartesian}
A \textbf{subcartesian space} is a Hausdorff Sikorski space $X$ such that every point $x\in X$ has a neighbourhood diffeomorphic to a subset of some $\R^n$.
\end{definition}

\begin{remark} Subcartesian spaces, introduced by Aronszaijn \cite{aronszajn} are sometimes called locally affine Sikorski spaces in the literature, or are defined in other (equivalent) ways. The precise connection between the various definitions is made clear in \cite{Watts}. A brief historical overview on subcartesian spaces can also be found in \cite {Watts-MSc}.
\end{remark}


\begin{proposition}[Second variation of Main Theorem 1]\label{Prop-Var2} The orbit space $X$ of a proper Lie groupoid $\G\arrows M$ is a subcartesian space.

If two Lie groupoids $\G$ and $\H$ are Morita equivalent, their orbit spaces $X$ and $Y$ are isomorphic as subcartesian spaces.
\end{proposition}
\begin{proof} First of all, as a quotient, $X$ can be endowed with the quotient Sikorski space structure $(X,C^\infty(X))$. The topology induced by $C^\infty(X)$ coincides with the quotient topology because $X$ is a normal space, so it is in the conditions of Proposition \ref{prop-toplogiescoincide}.

We have already seen (e.g. in the proof of Theorem \ref{Xissmooth}) that $(X,C^\infty(X))$ is locally isomorphic to the quotient of a representation of a compact group which, by Schwarz's theorem (Theorem \ref{thm-schwarz}), is isomorphic to a subspace of some $\R^n$. Hence $(X,C^\infty(X))$ is subcartesian.

If two Lie groupoids $\G$ and $\H$ are Morita equivalent, their orbit spaces $X$ and $Y$ are homeomorphic (Lemma \ref{lemma-moritarmk}); the algebra of smooth functions can be seen as the algebra of global sections of the structure sheaf, which is invariant under Morita equivalences (Theorem \ref{Xissmooth}).
\end{proof}

\section{Orbispaces as stratified spaces}
\label{chp4-stratifications}

In this section we study the canonical decomposition of the base and orbit space of a proper Lie groupoid. The decomposition is by smooth pieces that fit together in a prescribed way, as a stratification.

We start by recalling some of the general theory of stratifications associated with proper Lie group actions. Most of this material is rather classical, but we try to clarify some points where the literature can sometimes be confusing. Some references for this exposition are for example \cite{DK,Mather,Pflaum}. We then extend some of the theory of proper Lie group actions to proper Lie groupoids, obtaining in this way some of the results from \cite{hessel} and a principal type theorem for proper Lie groupoids (Theorem \ref{thm-principal_type}).
Throughout this section, we assume that $X$ is a connected topological space and $M$ is a connected manifold of dimension $n$.

\subsection{Stratifications}

\begin{definition}\label{dfn-stratification} Let $X$ be a Hausdorff second-countable paracompact space. A \textbf{stratification} of $X$ is a locally finite partition $\mathcal{S}= \{X_i \ |\ i\in I\}$ of $X$ such that its members satisfy:
\begin{enumerate} \item[1.] Each $X_i$, endowed with the subspace topology, is a locally closed, {\it connected} subspace of $X$, carrying a given structure of a smooth manifold; \item[2.] (frontier condition) the closure of each $X_i$ is the union of $X_i$ with members of $\cS$ of strictly lower dimension.
\end{enumerate}The members $X_i\in \mathcal{S}$ are called the \textbf{strata} of the stratification.
\end{definition}

We will study in detail some stratifications associated with proper Lie group actions and proper Lie groupoids, but let us start with some very simple examples.

\begin{enumerate}\item[1.] Any connected manifold comes with the stratification by only one stratum.
\item[2.] A manifold with boundary can be stratified by its interior and the connected components of the boundary.
\item[3.] If $M$ is compact, then the cone on $M$, \[CM=[0,1)\x M\ /\ \{0\}\x M\] comes with a stratification with two strata: the vertex point and $(0,1)\x M$.
\end{enumerate}

\begin{remark}(Comments on the definition and comparison with the literature)
When $X$ is actually a smooth manifold, it is usual to require that the strata are submanifolds of $X$. Similarly, when $X$ can be equipped with some sort of smooth structure, for example the ones described in Section \ref{chp3-smooth structures}, then it is natural to require some sort of compatibility between the smooth structure and the stratification. Section \ref{chp5-smooth_stratifications} is centred around this interplay.

Across the literature, it is possible to find quite a lot of variations on the definition of a stratification, typically so that the definition is most adapted to the problem under study. For example, some authors do not require $X$ to be Hausdorff, paracompact, or second-countable. The two main conditions used here that are often not mentioned are connectedness of the strata (which is discussed in detail in Remark \ref{rmk-connected_strata}) and the requirement that strata included in the closure of another stratum have strictly lower dimension. Although the latter condition is often not required, without it we would be forced to consider pathological examples (e.g. the closed topologist's sine curve, and even more pathological ones - see \cite[Ex.\ 1.1.12]{Pflaum}) that do not occur anyway in our study of proper actions and proper Lie groupoids. On the other hand, some authors require further conditions on how the strata fit together, for example the conditions of topological local triviality, or the cone condition; these will be discussed in Section \ref{chp5-smooth_stratifications}.
\end{remark}

\begin{remark}[On the condition of connectedness of the strata]\label{rmk-connected_strata}
The condition of connectedness of the strata is important not only as a technical condition but also conceptually. It is often present in the literature only implicitly, built into the definition of the partition that is to be studied. More precisely, one starts with a locally finite partition of $X$ by locally closed submanifolds $\cP$ and then one passes to the partition $\cP^c$ by connected components of $\cP$. The partition $\cP^c$ is then checked to satisfy the frontier condition. One of the usual motivations for passing to connected components is that the elements of $\cP$ might have components of different dimension. However, there is a much more fundamental reason to pass to connected components: in many important examples (such as the partition by orbit types - see Definition \ref{dfn-orbit-types}) the condition of frontier might not be satisfied unless we pass to connected components - see Example \ref{counterexample_not_connected_strata}.

At a more conceptual level, the condition of connectedness also allows for a global implementation of Mather's approach to stratifications using germs of submanifolds (see e.g. \cite{Mather,Pflaum}), but without making reference to germs. We explain below the precise connection with Mather's approach. The main point to do so, present in the next lemma, is to understand when two partitions may give rise, after passing to connected components, to the same stratification.

\end{remark}



\begin{lemma}\label{stratif-lemma-1} Let $\cP_{i}$, $i\in \{1, 2\}$, be two partitions of $X$ by smooth manifolds (whose connected components may have different dimensions) with the subspace topology; denote by $\cP_{i}^{c}$ the new partition obtained by taking the connected components of the members of $\cP_i$. Then $\cP_{1}^{c}= \cP_{2}^{c}$ if and only if, for each $x\in X$, there exists an open neighbourhood $U$ of $x$ in $X$ such that
\[ P_1\cap U= P_2\cap U,\]
where $P_i\in \cP_i$ are the members containing $x$.
\end{lemma}

\begin{proof} For the direct implication, let $x\in X$ and let $A_i\in \cP_{i}^{c}$ such that $x\in A_i$. Let $P_i\in \cP_i$ be the members containing $x$. Then there are open subsets $U_i$ of $X$ such that $U_i$ contains $A_i$ but not the other connected components of $P_i$. The open neighbourhood $U=U_1\cap U_2$ satisfies the second condition of the statement. 

To prove the converse implication, it suffices to show that, for $A_i\in \cP_{i}^{c}$ ($i\in \{1, 2\}$) with $A_1\cap A_2\neq \emptyset$, one must have $A_1= A_2$. Let $P_i\in \cP_i$ so that $A_i$ is a connected component of $P_i$. We first show that $A_1\cap A_2$ is open in $A_1$. Let $a\in A_1\cap A_2$. By hypothesis, we find a neighbourhood $U$ of $a$ so that $U\cap P_1= U\cap P_2$. Since $A_1$ is locally connected, we may assume that $U\cap A_1$ is connected. Then, since $U\cap A_1\subset U\cap P_1= U\cap P_2\subset P_2$, we know that $U\cap A_1$ sits inside a connected component of $P_2$. Since $a\in U\cap A_1$ and $A_2$ is the connected component of $P_2$ containing $a$, we must have 
$U\cap A_1\subset A_2$, hence $U\cap A_1\subset A_1\cap A_2$. This proves that $A_1\cap A_2$ is open in $A_1$. Note that this implies that $\{A_1\cap B\ |\ B\in \cP_{2}^{c}\}$ is a partition of $A_1$ by open subspaces hence, by the connectedness of $A_1$, it must coincide with one of the members of this family - and that is necessarily the non-empty $A_1\cap A_2$. Hence $A_1\subset A_2$ and the reverse inclusion is proved similarly.
\end{proof}

\begin{definition} A \textbf{decomposition} of a Hausdorff second-countable topological space $X$ is a partition $\cP$ satisfying all conditions from Definition \ref{dfn-stratification} except possibly connectedness of the strata.
\end{definition}

\begin{example} Some decompositions cannot be made into a stratification in our sense by passing to connected components. For example, consider the decomposition of the plane $\R^2$ into three pieces $A$, $B$, and $C$: $A$ equals the origin $\{0\}$,  $B$ equals the union of all circles centred at the origin, of radius equal to $1/n$, with $n\in \mathbb{N},$ and $C=\R^2\backslash A\cup B$. Passing to connected components, we would lose local finiteness of the partition.
\end{example}

Mather's approach using germs leads to the following alternative definition of stratification (cf. \cite{Mather,Pflaum}), that we designate by germ-stratification.

\begin{definition} A \textbf{germ-stratification} of a topological space $X$ is a rule which assigns to each $x\in X$ a germ $\S_x$ of a closed subset of $X$, such that, for each $x\in X$, there is a neighbourhood $U$ of $x$ and a decomposition $\cP$ of $U$, with the property that for all $y\in U$, $\S_y$ is the germ of the piece of $\cP$ containing $y$. 
\end{definition}

Given any decomposition $\S$ on $X$, we can produce a germ-stratification by assigning to each $x\in X$ the germ of the piece of the decomposition containing $x$. A result of Mather \cite[Lemma 2.2]{Mather} states that any germ-stratification arises in this way. Lemma \ref{stratif-lemma-1} guarantees that as long as we restrict to those decompositions that are stratifications and their corresponding germ-stratifications, this correspondence is indeed bijective. Accordingly, germ-stratifications are usually simply called stratifications in the literature.

\begin{definition}Given a stratification $\S$ there is a natural partial order on the strata given by \[S\leq T\ \Leftrightarrow\ S\subset \overline{T}.\] The union of all maximal strata (with respect to this order) forms a subspace $M^{\cS-\mathrm{reg}}\subset M$ called the \textbf{$\cS$-regular part of $M$}.
\end{definition}

 The following lemma shows that maximality of a stratum is a local condition (cf. \cite{PMCT3}).

\begin{lemma}\label{stratif-lemma-0} A stratum $S\in \cS$ is maximal if and only if it is open. The regular part  $M^{\cS-\mathrm{reg}}$ is open and dense in $M$.
\end{lemma}
\begin{proof} Assume that $S$ is a maximal stratum which is not open. Let $x\in S$ lie outside the interior of $S$. Choose a neighbourhood $V$ of $x$ in $M$ which intersects with only finitely many members of $\cS$. Since $x$ is not in the interior of $S$, by choosing
a sequence of neighbourhoods $V\supset V_0\supset V_1\supset \ldots $ that shrink to $x$,  we can find $x_n\in V_n\setminus S$ for each $n$. We obtain in this way a sequence $(x_n)_{n\geq 0}$ converging to $x$, with $x_n\in V\setminus S$. Each $x_n$ belongs to one of the finitely many members of $\cS$ which meets $V$, so after passing to a subsequence we may assume that $x_n\in T$ for all $n$, for some $T\in \cS$. It follows that $x\in \overline{T}$, hence $S\cap \overline{T}\neq \emptyset$, and so $S\subset \overline{T}$. From the maximality of $S$ we have that $S= T$, which contradicts the fact that $x_n$ is not in $S$. 

For the converse, assume that $S$ is open and $S\subset \overline{T}$ for some $T\in \cS$. If $S\neq T$, it follows that $S$ is a stratum of dimension strictly less than that of $T$, which cannot be the case since $S$ is open.

The regular part $M^{\cS-\mathrm{reg}}$ is clearly open, being a union of open strata. Given an arbitrary $x\in M$, it belongs to at least one stratum; consider a strict chain $x\in S_1 < S_2< \ldots < S_k$ which cannot be continued. Then $S_k$ is maximal and $x\in \overline{S}_k$, hence $x$ is in the closure of $M^{\cS-\mathrm{reg}}$, proving that this space is dense. 
\end{proof}

Some natural questions about the regular part of $M$ come to mind; how different is it from $M$? Is it connected? The following lemma tries to partially address these questions.




\begin{lemma}\label{stratif-lemma-4}  Let $\cS$ be a stratification on a smooth manifold $M$, with no strata of codimension $1$. Then the $\cS$-regular part of $M$, denoted by $M^{\mathrm{reg}}$, is connected.
\end{lemma} 

\begin{proof}
Let $x$ and $y$ be two points of $M^{\mathrm{reg}}$ and consider a smooth curve $\gamma:[0,1]\to M$ connecting $x$ and $y$ (recall that by $M$ is connected by assumption). The image of $\gamma$ is compact, so it can be covered by a finite number of open subsets of $M$, each of which intersects finitely many strata. Let $U$ be the union of those open subsets. 
Then by the Transversality homotopy theorem (cf. \cite[p.\ 70]{GP}), it is possible to find a map $\gamma':[0,1]\to U$ which is homotopic to $\gamma$ and transverse to all the finitely many strata of codimension greater than $1$ in $U$, which means that it misses them. Since there are no strata of codimensions $1$, the image of the map $\gamma'$ must be completely contained in the union of the strata of codimension $0$, which is precisely $M^{\mathrm{reg}}$.
\end{proof}


\subsection{Proper group actions: the canonical stratification.}\label{sec-proper_actions_strat}

We recall an important example of a stratification, associated to a proper action of a Lie group $G$ on a manifold $M$. This example serves as both motivation and background for the study of stratifications on proper groupoids.  We take the standard approach of first defining a natural partition $\cP$ on $M$ associated to the action and then passing to the partition by connected components $\cP$. For the whole of this section, let $G$ be a Lie group acting properly on a smooth manifold $M$.


\begin{definition}\label{dfn-orbit-types} The \textbf{orbit type equivalence} is the equivalence relation on $M$ given by
\[ x\sim y\  \Longleftrightarrow \  G_x\sim G_y \ \ (\textrm{i.e.} \ G_x\ \textrm{and}\ G_y\ \textrm{are\ conjugate\ in} \ G) .\]
The \textbf{partition by orbit types}, denoted by $\cP_{\sim}(M)$, is the  resulting partition  (each member of $\cP_{\sim}(M)$ is called an \textbf{orbit type}).
\end{definition}

The reason for the terminology is that $x\sim y$ is equivalent to the fact that the orbits through $x$ and $y$ are diffeomorphic as $G$-manifolds. The members of this partition can be indexed by conjugacy classes $(H)$ of subgroups $H$ of $G$. To each such conjugacy class corresponds the orbit type \[M_{(H)}=\{x\in M\ |\ G_x\sim H \} \in \cP_{\sim}(M).\]

Points in the same orbit belong to the same orbit type, so the partition $\cP_{\sim}(M)$ descends to a partition by orbit types $\cP_{\sim}(M/G)$ of the orbit space. 
Passing to the connected components of the members of  $\cP_{\sim}(M)$ and of $\cP_{\sim}(M/G)$, we obtain stratifications of $M$ and of $M/G$. Moreover, the projection of the strata on $M$  by the quotient map are the strata on $M/G$. In other words, passing to the quotient commutes with taking connected components of the orbit types. A proof of these facts can be found for example in \cite{DK,Pflaum}.

\begin{definition} \textbf{The canonical stratification} on $M$ (respectively $M/G$) associated to the action of $G$ is the partition of $M$ (respectively of $M/G$) by connected components of the members of  $\cP_{\sim}(M)$ (respectively $\cP_{\sim}(M/G)$) and is denoted by $\cS_{G}(M)$ (respectively $\cS(M/G)$).
\end{definition}

There are other partitions of $M$ that also induce the same stratification $\cS(M/G)$, by passing to connected components. Let us start with the simplest one to describe.

\begin{definition}\label{dfn-isotropy-iso-type} The \textbf{partition by isotropy  isomorphism classes} on $M$, denoted by $\cP_{\cong}(M)$, is defined by the  equivalence relation on $M$ given by
\[ x\cong y\ \Longleftrightarrow \ G_x\cong G_y\  (\textrm{Lie\ group\ isomorphism}) .\]
\end{definition}

The members of this partition can be indexed by isomorphism classes $[H]$ of subgroups $H$ of $G$. To each such conjugacy class corresponds the orbit type \[M_{[H]}=\{x\in M\ |\ G_x\cong H \} \in \cP_{\cong}(M).\]

\begin{proposition}\label{prop-also-induces-strat}
After passing to connected components, $\cP_{\cong}(M)$ induces the same stratification on $M$ as $\cP_{\sim}(M)$.
\end{proposition}
\begin{proof}
Using Lemma \ref{stratif-lemma-1}, it is enough to check that given any $x\in M$, there is an open neighbourhood $U$ of $x$ such that $M_{[G_x]}\cap U = M_{(G_x)}\cap U$. Using the normal form $[G\x_{G_x}\N_x]$ around $x$ given by the Tube theorem (Theorem \ref{thm-tube}), and since points belonging to the same orbit are equivalent for both $\cong$ and $\sim$, it is enough to compare $\cong$ and $\sim$ for points on a neighbourhood of $x$ in $\N_x$. Therefore, we have reduced the problem to checking that given a representation $V$ of a compact subgroup $K\subset G$, it holds that $V_{[K]}=V_{(K)}$. The following lemma guarantees that this is always the case.
\end{proof}

\begin{lemma}\label{stratif-lemma-2} If $H$ is a closed subgroup of a compact Lie group $K$ with the property that $H$ is isomorphic (or just diffeomorphic) to $K$, then $H= K$.
\end{lemma}

\begin{proof}
For dimensional reasons we see that $H$ and $K$ must have the same Lie algebra; from this it follows that their connected components containing the identity, $H^0$ and $K^0$, coincide. The fact that $H$ is diffeomorphic to $K$ implies that they also have the same (finite) number of connected components, from which the statement follows.
\end{proof}

Yet another partition that induces the canonical stratification is the partition by \textbf{local types} (cf. \cite[Def.\ 2.6.5]{DK}).
Its members are indexed by equivalence classes of pairs $(H, V)$, where $H$ is a subgroup of $G$ and $V$ is a representation of $H$; two such pairs $(H, V)$ and $(H', V')$ are equivalent if $H$ is conjugate to $H'$ by some $g\in G$ and $V\cong V'$ by an isomorphism compatible with $\mathrm{Ad}_{g}$.

\begin{definition} The \textbf{partition by local types} on $M$, denoted by $\cP_{\approx}(M)$, is defined by the  equivalence relation on $M$ given by
\[ x\approx y\ \Longleftrightarrow \  (G_x, \mathcal{N}_x)\approx (G_y, \mathcal{N}_y) ,\]
where $\mathcal{N}_x$ is the normal representation at $x$.
\end{definition}

The reason for the terminology is that, as a consequence of the Tube theorem (Theorem \ref{thm-tube}), $x$ and $y$  belong to the same local type if and only if the orbits through $x$ and $y$ admit equivariantly-diffeomorphic neighbourhoods.
By the same arguments as in the proof of Proposition \ref{prop-also-induces-strat} above, we conclude the following.

\begin{proposition}
After passing to connected components, $\cP_{\approx}(M)$ induces the canonical stratification on $M$.
\end{proposition}

The difference between the three partitions discussed is that they group the strata of $\cS_G(M)$ in different ways: it is easy to see that
\[ \cP_{\cong}(M) \prec \cP_{\sim}(M) \prec \cP_{\approx}(M)  \]
in the sense that each member of $\cP_{\cong}(M)$ is a union of members of $\cP_{\sim}(M)$, etc.
The obvious inclusions of a member of $\cP_{\sim}(M)$ into the corresponding member of $\cP_{\cong}(M)$, and of a member of $\cP_{\approx}(M)$ into the corresponding member of $\cP_{\sim}(M)$ are strict in general, as the following example shows.


\begin{example}\label{counterexample_not_connected_strata}Consider the finite group $\mathbb{Z}_2\times  \mathbb{Z}_2= \{\pm 1\}\times \{\pm 1\}$, acting on $\mathbb{R}P^2$  by $(\epsilon, \eta)\cdot [x: y: z]= [x: \epsilon y: \eta z]$. 
The subgroups $\mathbb{Z}_2\times \{1\}$ and $\{1\}\times  \mathbb{Z}_2$ are isomorphic but not conjugate, and arise as isotropy groups of the points $[1:0:1]$ and $[1:1:0]$ respectively. Hence \[\mathbb{R}P^2_{(\mathbb{Z}_2\times \{1\})}\neq \mathbb{R}P^2_{[\mathbb{Z}_2\times \{1\}]}.\]

On the other hand, all fixed points of a given action have the same orbit type, but the local types of fixed points can differ. For example, consider the action of the circle $S^1$ on $\mathbb{C}P^2$ given by $\theta \cdot [x:y:z]=[z:\theta^2 y: \theta^5 z]$.
The isotropy representation at the fixed points $[1:0:0]$, $[0:1:0]$ and $[0:0:1]$ are isomorphic with the representation of $S^1$  on $\mathbb{C}^2$ of weights $(2,5)$, $(-2,3)$ and $(-5,-3)$ respectively (we say that a representation of $S^1$ on $\mathbb{C}^2$ has weight $(m,n)\in \mathbb{Z}^2$ if it is given by $\theta\cdot(z,w) = (\theta^m z,\theta^n w)$).

\end{example}

\begin{remark}It is important to note that, as mentioned in Remark \ref{rmk-connected_strata}, the passage to connected components of $\cP_{\sim}(M)$ is really necessary to guarantee that we end up with a stratification. Besides the obvious problem that orbit types may be disconnected, there is the more serious issue that the frontier condition may not be satisfied (between orbit types). For example, there may be orbit types whose closure contains some, but not all of the fixed points of the action. This is the case in Example \ref{counterexample_not_connected_strata} above:
For the subgroup $H=  \mathbb{Z}_2\times \{1\}$, we see that \[\mathbb{R}P^2_{(\mathbb{Z}_2\times \{1\})}=\{[x:0:z]\in \mathbb{R}P^2\ |\ x\neq 0\neq z\},\] so its closure contains the fixed points $[1:0:0]$ and $[0:0:1]$, but not the fixed point $[0:1:0]$.

Unlike the case for the two partitions discussed before, all the connected components of a local type have the same dimension (which can be seen using the normal form given by the Tube theorem). However, $\cP_{\approx}(M)$ may still fail to satisfy the frontier condition: the same counterexample as for $\cP_{\sim}(M)$ works here as well since $[0:1:0]$ and $[0:0:1]$ belong to the same local type.
\end{remark}

 We recall another interesting stratification on $M$, which appears as the infinitesimal version of the canonical stratification from the previous section. The idea is that replacing the isotropy Lie groups $G_x$ by their Lie algebras $\mathfrak{g}_x$, one obtains similar (but in general different) partitions of $M$.

\begin{definition}The \textbf{infinitesimal orbit type equivalence} is the equivalence relation on $M$ given by
\[ x{\sim}_{\mathrm{inf}}\ y\ \Longleftrightarrow \  \mathfrak{g}_x\sim \mathfrak{g}_y \ \ (\textrm{i.e.} \ \mathfrak{g}_x\ \textrm{and}\ \mathfrak{g}_y\ \textrm{are\ conjugate\ in} \ \mathfrak{g}) .\]
 The \textbf{partition by infinitesimal orbit types} is the resulting partition, denoted by $\cP_{{\sim}_{\mathrm{inf}}}(M)$ (each member of  $\cP_{{\sim}_{\mathrm{inf}}}(M)$ is called an \textbf{infinitesimal orbit type}).

The \textbf{infinitesimal canonical stratification} on $M$, denoted by $\cS_{G}^{\mathrm{inf}}(M)$, is the partition of $M$ by connected components of the members of $\cP_{{\sim}_{\mathrm{inf}}}(M)$.
\end{definition}

Similarly, we define infinitesimal versions of the partitions by isotropy isomorphism classes and by local types - define the equivalence relations ${\approx}_{\mathrm{inf}}$ and ${\cong}_{\mathrm{inf}}$ by replacing the isotropy Lie groups $G_x$ by their Lie algebras $\mathfrak{g}_x$ in the definitions. The infinitesimal analogue of Lemma \ref{stratif-lemma-2} is obvious, and therefore we obtain that the partitions $\cP_{{\sim}_{\mathrm{inf}}}(M)$ and $\cP_{{\cong}_{\mathrm{inf}}}(M)$ induce the infinitesimal canonical  stratification.

\begin{remark} It is easy to see that points in the same orbit belong to the same infinitesimal orbit type, so we obtain a partition by infinitesimal orbit type on the orbit space $M/G$. However, this partition does not in general induce a stratification. Indeed, the members of the partition on the orbit space may fail to be manifolds. For example, in the case of the action of $\mathbb{Z}_2$ on $\R$ by reflection at the origin, all points have the same infinitesimal orbit type, but the orbit space is not a manifold.
\end{remark}


\subsection{Proper actions  - principal and regular types}

We have seen several partitions inducing the same stratifications $\cS_{G}(M)$ and $\cS_{G}^{\mathrm{inf}}(M)$. This allows us to use different partitions when proving results about the stratification, using whichever is more convenient for the proof. 

It is interesting to distinguish which notions are intrinsic to the stratification, and which are particular to one of the partitions giving rise to it; similarly, we can wonder about whether a given result on the stratification can be strengthened to a result on one of the partitions giving rise to it.

We now recall some properties of the maximal strata of the canonical stratification and point out the relation with the partitions that give rise to it.

\begin{definition}
The \textbf{principal part} of $M$ is defined to be the $\cS_{G}(M)$-regular part of $M$ and is denoted by $M^{\mathrm{princ}}:= M^{\cS_{G}(M)-\mathrm{reg}}$. The orbits inside $M^{\mathrm{princ}}$ are called \textbf{principal orbits}.
\end{definition}

In order to check whether a point $x\in M$ belongs to a principal orbit, we use Lemma \ref{stratif-lemma-0} and a tube $G\times_{G_x} V$ around $x$. We arrive at the condition that $V^{G_x}$ is open in $V$, which leads to the following characterization.

\begin{lemma}\label{G-str-prin-pts} For any point $x\in M$, the following are equivalent:
\begin{enumerate}
\item[1.] $x\in M^{\mathrm{princ}}$ 
\item[2.] the action of $G_x$ on the normal space to the orbit is trivial.
\end{enumerate}
In this case all the orbits $G\cdot y$ through points $y$ close to $x$ are diffeomorphic $G\cdot x$. 
\end{lemma}

By definition, $M^{\mathrm{princ}}$ is intrinsically associated to the canonical stratification. But to understand it better, we recall a related notion, defined in terms of the partition $\cP_{\sim}(M)$ by orbit types. First of all, note that there is a partial order on the orbit types that is analogous to the ordering on the strata:
\[ M_{(H)} \geq M_{(K)}  \Longleftrightarrow K\ \textrm{is}\ G-\textrm{conjugate\ to \ a\ subgroup\ of}\ H .\]

The maximal orbit types (with respect to this order) are called principal orbit types. They are related to $M^{\mathrm{princ}}$ by the Principal orbit type theorem (cf. \cite[Thm.\ 2.8.5]{DK}, or Subsection \ref{subsection_principal} for a  generalization of this theorem for proper Lie groupoids). The theorem states that $\cP_{\sim}(M)$ admits one and only one maximal orbit type: there exists a unique conjugacy class, denoted  $(H_{\mathrm{princ}})$ such that any isotropy group $G_x$ of the action contains a 
conjugate of $H_{\mathrm{princ}}$. In terms of the stratification, this means that \[ M^{\mathrm{princ}}= M_{(H_{\mathrm{princ}})}.\]
Hence the maximal strata of $\cS_{G}(M)$ are precisely the connected components of the principal orbit type. We also see that, although it was originally defined in terms of $\cP_{\sim}(M)$, the notion of principal orbit type only depends on the stratification.

Moreover, the Principal orbit type theorem also states that, even when $M^{\mathrm{princ}}$ is not connected, the quotient $M_{(H_{\mathrm{princ}})}/G$ is connected. Hence the stratification $\cS(M/G)$ of the quotient $M/G$ has one and only one maximal (principal) stratum.


Alternatively, we could proceed similarly but using $\cP_{\cong}(M)$ instead of $\cP_{\sim}(M)$; in that case, the partial order to consider is 
\[ M_{[H]} \geq M_{[K]}  \Longleftrightarrow K\ \textrm{is\ isomorphic\ to \ a\ subgroup\ of}\ H,\]
which is, indeed, a partial order by Lemma \ref{stratif-lemma-2}. The corresponding version (for $\cP_{\cong}(M)$) of the Principal orbit type theorem results in a unique maximal element, which is precisely $M_{[H_{\mathrm{princ}}]}$. 

\begin{proposition} The maximal elements of $\cP_{\cong}(M)$ and of $\cP_{\sim}(M)$ coincide, i.e., \[M_{[H_{\mathrm{princ}}]} = M_{(H_{\mathrm{princ}})}.\]
\end{proposition}

\begin{proof}
It always holds that $M_{(H)}\subset M_{[H]}$, so we are left with checking the converse. If $x \in M_{[H_{\mathrm{princ}}]}$ we know that $G_x\cong H_{\mathrm{princ}}$, but we also know that $G_x$ is conjugate to a subgroup of  $H_{\mathrm{princ}}$. Using  Lemma \ref{stratif-lemma-2}, we see that the subgroup must be the entire $H_{\mathrm{princ}}$, and so $G_x$ is conjugate to it; hence $x$ is also in $M_{(H_{\mathrm{princ}})}$. 
\end{proof}

Finally, let us mention that, using the partition by local types, one would equally find that there is a single maximal local type, which again coincides with $M^\mathrm{princ}$ (cf. \cite[Cor.\ 2.8.6]{DK}).

Let us now focus on the infinitesimal canonical stratification. The associated regular part is denoted: 
\[ M^{\mathrm{reg}}:= M^{\cS_{G}^{\mathrm{inf}}(M)-\mathrm{reg}}.\]
In complete similarity with Lemma \ref{G-str-prin-pts}, using a the normal form given by the Tube theorem (Theorem \ref{thm-tube}), we find:

\begin{lemma}\label{G-str-reg-pts} For any point $x\in M$, the following are equivalent:
\begin{enumerate}
\item[1.]  $x\in M^{\mathrm{reg}}$;
\item[2.]  the infinitesimal action of $\gg_{x}$ on the normal space to the orbit is trivial;
\item[3.]  the action of $G_{x}^{0}$ on the normal space to the orbit is trivial;
\item[4.]  all the orbits through points close to $x$ have the same dimension, 
\end{enumerate}
In this case all the orbits $G\cdot y$ through points $y$ close to $x$ are coverings of $G\cdot x$;
\end{lemma}


The next proposition shows that, although the infinitesimal canonical stratification behaves worse than the canonical stratification (e.g. it does not induce a stratification on the quotient), it does have some  advantages over $\cS_{G}(M)$.

\begin{proposition}\label{codim1_strata_action} The infinitesimal canonical stratification satisfies:

\begin{itemize}
\item[1.] $\cS_{G}^{\mathrm{inf}}(M)$ does not contain codimension 1 strata.
\item[2.] $M^{\mathrm{reg}}$ is connected. 
\item[3.] $\cS_{G}^{\mathrm{inf}}(M)$ has one and only one maximal strata. 
\end{itemize}
\end{proposition}

\begin{proof} Assume that $S$ is a stratum of codimension one; $S$ is a connected component of a subspace of type
\[ M_{[\mathfrak{h}]}= \{ x\in M\ |\ \gg_x\cong \mathfrak{h} \}\]
for some Lie subalgebra $\mathfrak{h}$ of $\gg$. Let $x\in S$; we may assume that $G_x= H$.  Let $U\subset G\times_{H} V$ be a tube around $x$, with $x$ represented by $(e, 0)$. Then 
\[ U\cap M_{[\mathfrak{h}]} \cong \{ y= [a, v]\in G\times_{H} V\ |\ G\times_{G_x} \mathfrak{h}_v= h\}= G\times_{H} V^{\mathfrak{h}} .\]
To achieve codimension $1$, $V^{\mathfrak{h}}$ must be of codimension $1$ in $V$. Let $W$ be the complement of $V^{\mathfrak{h}}$ with respect to an $H$-invariant metric on $V$; then $W^{\mathfrak{h}}= 0$, hence $W$ is a non-trivial  one dimensional representation of the compact connected Lie group $H^0$, which is impossible. The fact that part 1 implies part 2 follows from Lemma \ref{stratif-lemma-4}; part 3  is just a reformulation of part 2.  
\end{proof}

\subsection{Morita types}\label{sec-Morita_types}

In this section we introduce the canonical stratification associated to a proper Lie groupoid. This generalizes the canonical stratification induced by a proper Lie group action.
As in the case of proper actions, a proper Lie groupoid induces a stratification not only on its base, but also on the orbit space. Let us mention already that one of the essential properties of the stratification on the orbit space is that it is Morita invariant (see Remark \ref{rmk-morita-inv-strat}), meaning that it is intrinsically associated to the orbispace presented by the groupoid.

Let $\G$ be a proper Lie groupoid over $M$ and denote its orbit space by $X$.

\begin{definition}\label{dfn-morita-types} The \textbf{Morita type equivalence} is the equivalence relation on $M$ given by
\[ x\sim_\M y\  \Longleftrightarrow   \ (\G_x, \mathcal{N}_x)\cong (\G_y, \mathcal{N}_y),\]
where $(\G_x, \mathcal{N}_x)\cong (\G_y, \mathcal{N}_y)$ means that there is an isomorphism $\phi:\G_x \to \G_y$ and a compatible isomorphism of representations between the normal representations $\N_x$ and $\N_y$.

The \textbf{partition by Morita types}, denoted by $\cP_{\M}(M)$, is defined to be the resulting partition. Each member of $\cP_{\M}(M)$ is called a \textbf{Morita type}.
\end{definition}

In other words, the partition by Morita types is indexed by equivalence classes of pairs $(H, V)$ where $H$ is a Lie group, $V$ is a representation of $H$, and two pairs are equivalent in the way described above: $H\cong H'$ and $V\cong V'$ in a compatible way. In this case, we write $(H, V)\cong (H', V')$. 
Set $[H,V]$ for the equivalence class of the pair $(H, V)$. If  $[H,V]=\alpha$, then the element of $\mathcal{P}_\M(M)$ corresponding to $\alpha$ is 
\[M_{(\alpha)}=\{x\in M \ |\ [\G_x,\N_x]=\alpha\}\in \mathcal{P}_\M(M).\]
We also denote by $M_{(x)}$ the Morita type of a point $x\in M$.

\begin{remark}\label{rmk-morita-inv-strat}
The Morita type of a point $x\in M$ depends only on the Morita equivalence class of the local model  $\N(\G_{\O_x})$ of the groupoid in a neighbourhood of its orbit ${\O_x}$. Indeed, the local models $\N(\G_{\O_x})$  and $\N(\G_{\O_y})$ around $x$ and $y$ are Morita equivalent to $\G_x\lx\N_x$ and $\G_y\lx\N_y$, respectively. These in turn are Morita equivalent to each other if and only if there is an isomorphism $\phi:\G_x \to \G_y$ and an isomorphism of representations between $\N_x$ and $\N_y$, compatible with $\phi$.

Therefore, points in the same orbit belong to the same Morita type and hence we also obtain a \textbf{partition by Morita types on the orbit space}, $\cP_\M(X)$. The projection map $\pi:M\to X$ takes Morita types in $M$ to Morita types in $X$; we use the notation $X_{(\alpha)}=\pi(M_{(\alpha)})$ for Morita types in the orbit space.
This partition on $X$ is really associated to the orbispace $X$ presented by $\G$ and not to $\G$ itself. Indeed, by its very definition, the Morita type of $\O\in X$ only depends on Morita invariant information.
\end{remark}

\begin{definition}
The \textbf{canonical stratification} on $M$, denoted by $\cS_\G(M)$, is the partition on $M$ obtained by passing to connected components of $\cP_{\M}(M)$. 
The \textbf{canonical stratification on the orbit space} $X$, denoted by $\cS(X)$, is the partition on $X$ obtained by passing to connected components of $\cP_\M(X)$. 

In Section \ref{sec:gpd_canonical_strat} we see  that these partitions are, indeed, stratifications.
\end{definition}

\begin{remark}By passing to the corresponding germ-stratification, $\cS(X)$ corresponds to the canonical germ-stratification of \cite{hessel}.
\end{remark}

\subsection{Comparison with other equivalence relations}

The notion of partition by isotropy isomorphism classes (see Definition \ref{dfn-isotropy-iso-type}) still makes sense for a general Lie groupoid, so we can compare the partitions $\cP_{\cong}(M)$ and $\cP_\M(M)$ (see Definition \ref{dfn-morita-types}), and we see that $\cP_{\cong}(M) \prec \cP_{\M}(M)$. In the case of an action groupoid of a proper Lie group action, it also makes sense to compare these partitions with the ones by orbit types and by local types.
It is easy to see that  \[ \cP_{\cong}(M) \prec \cP_{\M}(M) \prec \cP_{\approx}(M), \] and Example \ref{counterexample_not_connected_strata} can be used to check that these comparisons are strict.
There is no such relation comparing $\cP_\M(M)$ and $\cP_\sim(M)$ in general: for the $\mathbb{Z}_2\x\mathbb{Z}_2$ - action of Example \ref{counterexample_not_connected_strata}, the Morita types of $[1:0:1]$ and $[1:1:0]$ are the same, while their orbit types are different. On the other hand, all fixed points for the $S^1$ - action of Example \ref{counterexample_not_connected_strata} have the same orbit type, but their Morita types are different.

Nonetheless, as soon as we pass to connected components, Morita types induce the same stratification as the other partitions.

\begin{proposition}\label{prop-morita_types_action_groupoid}
Let $\G$ be the action groupoid associated to a proper Lie group action. The partition by Morita types on $M$ induces, after passing to connected components, the canonical stratification associated with the action. 
\end{proposition}
\begin{proof} 
The situation here is in complete analogy to the comparison between $\cP_{\sim}(M)$ and $\cP_{\cong}(M)$: apply Lemma \ref{stratif-lemma-1} and use the normal form given by the Tube theorem (Theorem \ref{thm-tube}) to compare $\cP_{\M}(M)$ with $\cP_{\approx}(M)$. Looking at the normal form means we consider the associated bundle $G\x_K V$, where $V$ is a representation $V$ of a compact subgroup $K\subset G$. We have to look at the points $v\in V$ with the property that $K_{v}$ is conjugate to $K$, (for $\approx$), or with the property that $K_v$ is isomorphic to $K$ (for $\sim_\M$), and additionally that the normal representation of $K_v$ on $\N_v$ is isomorphic to the representation of $K$ on $V$ in a compatible way with the isomorphisms of $K_v$ and $K$; once more, the first condition reduces to the condition $K_v= K$ because of Lemma \ref{stratif-lemma-2}. The conditions on the compatibility of the isomorphism of the representation become the same in both cases - compatibility with the identity map of $K$.
\end{proof}

The following result relates the Morita types for a proper Lie groupoid and the Morita types for the action groupoids given by the local model.

\begin{lemma}[Reduction to Morita types on a slice]\label{lemma-reducemorita}Let $\G\arrows M$ be a proper Lie groupoid and let $x\in M$.
Then there are invariant open sets $U$ around $x$ in $M$ and $W$ around $0$  in $\N_x$ (which we identify with a slice at $x$) such that the intersection of the Morita types for $\G$ with $U$ are given by the saturation of the Morita types for the linear action of $\G_x$ on $W$. 
\end{lemma}
\begin{proof} By the Linearization theorem for proper groupoids (Theorem \ref{lin}), there are invariant open subsets $U$ around $\O_x$ in $M$ and $W$ around $0$ in $\N_x$ such that $\G_U$ is Morita equivalent to the action groupoid $\G_x \ltimes W$. We identify $W$ with a slice $S$ at $x$ in $M$.
By this Morita equivalence, the Morita types in $U$ coincide with the saturation of the Morita types for the isotropy action of $\G_x$ on the slice $S$, and consequently correspond to the Morita types for the linear action of $\G_x$ on $W$.
\end{proof}

Using the notation from the previous lemma, let $F$ be the orthogonal complement to  the fixed point set $W^{\G_x}$ with respect to an invariant inner product, as in \cite{DK}.
 The isomorphism \[W^{\G_x}\x F\to W,\ \ (w,f)\mapsto w+f\] is equivariant.
  There is also an equivariant diffeomorphism \[\phi:\R_+\x \Sigma \to F\backslash \{0\},\] where $\Sigma$ denotes the unit sphere in $F$, with respect to the inner product used above; $\G_x$ acts on $\R_+\x \Sigma$ by $g(r,p)=(r,gp)$.

\begin{lemma}\label{reduceproblem}  Using the notation introduced above and identifying $W$ with a slice at $x$, the intersection of the Morita type of $x$ with $U$ is the $\G$-saturation of $W^{\G_x}$.
Each other Morita type is given by the $\G$-saturation of \[W^{\G_x} + \phi(\R_+\x T)\subset W,\] where $T$ is a Morita type for the action of $\G_x$ on $\Sigma$.
\end{lemma}

\begin{proof} The Morita type of the point $x$ (which is identified with $0\in W$) for the action of $\G_x$ on $W$ is given by $W_{(x)}=W^{\G_x}$. To see this, note that every point of $W_{(x)}$ has an isotropy group isomorphic to $\G_x$, hence equal to $\G_x$ by Lemma \ref{stratif-lemma-2}, so that $W_{(x)}\subset W^{\G_x}$. The converse inclusion can be deduced from the fact that all points of $W^{\G_x}$ have the same orbit type, hence the same Morita type. We can assume that this is the case since for a small enough open containing $0$, the members of $\cP_{\sim}(W)$ and  $\cP_\M(W)$ containing $0$ coincide. The first statement of the lemma follows by applying Lemma \ref{lemma-reducemorita}.

The second part of the statement is obtained as a consequence of Lemma \ref{lemma-reducemorita} and of the fact that the isomorphism $W^{\G_x}\x F\to W$ and the diffeomorphism $\phi:\R_+\x \Sigma \to F\backslash \{0\}$ are equivariant.
\end{proof}

\begin{remark}[Morita types on a neighbourhood of a point]\label{rmk-morita_types_around_point}  When we are only interested in how Morita types look like in a small neighbourhood of a point $x$ in the base, we can use the local model for $\G$ around $x$ given by Proposition \ref{prop-local_model_at_point}. This local model is the product of a pair groupoid $O\x O\arrows O$ with an action groupoid $\G_x\lx W\arrows W$, where $W$ is an invariant ball centred at $0$ in $\N_x$, on which the compact group $\G_x$ acts linearly. So we see that the intersection of each Morita type in $M$ with the neighbourhood $O\x W$ of $x$ is of the form $O\x T$ where $T$ is a  Morita type in $W$ for the action groupoid $\G_x\lx W\arrows W$.
\end{remark}

\subsection{The canonical (Morita type) stratifications}\label{sec:gpd_canonical_strat}

\begin{proposition}\label{onemoritatype} Let $\G\arrows M$ be a proper Lie groupoid with only one Morita type. Then the orbit space $X$ is a smooth manifold and the canonical projection $\pi:M\to X$ is a submersion, whose fibres are the orbits.  

\end{proposition}
\begin{proof} We already know from Proposition \ref{Xissmooth} that $(X,\CC(X))$ is a differentiable space, so in order to show that it is a smooth manifold it is enough to show that every point $\O\in X$ has a neighbourhood $U$ such that $(U,\CC(U))$ is a smooth manifold. Consider $\O\in X$ and $x\in M$ such that $\pi(x)=\O$. We have seen in the proof of  Proposition \ref{Xissmooth} that $\O$ has a neighbourhood $U$ in $X$ such that
\[(U,\CC(U))\cong (W/\G_x, \CC(W)^{\G_x}),\]
where $W$ is an open in $\N_x$ containing the origin, which is invariant for the isotropy representation of $\G_x$ on $\N_x$. To avoid confusion let us use the notation $H:=\G_x$. By Morita equivalence, since $\G$ only has one Morita type, the same holds for $\G_x\lx W$. This implies that for every $v\in W$, the isotropy group $H_v$ is isomorphic to $H$, hence $H_v=H$, by Lemma \ref{stratif-lemma-2}. This means that $H$ acts trivially on $W$ and so $$(U,\CC(U))\cong (W/\G_x, \CC(W)^{\G_x})=(W,\CC(W))$$ is a smooth manifold, and so is $X$.

To check that $\pi$ is a submersion, as this is a local property, we can use a neighbourhood $V\subset P_x\x_{\G_x}W$ of $x$ in the local model. Since $\G_x$ acts trivially on $W$ then $P_x\x_{\G_x}\N_x=\O\x W$. The restriction of the projection $\pi$ to $V$ is then given by $\pi(y,v)=v$, hence $\pi$ is a submersion.
\end{proof}

We denote the restriction of $\G$ to a Morita type $M_{(\alpha)}$ by $\G_{(\alpha)}:=s^{-1}(M_{(\alpha)})$. It is clear that $\G_{(\alpha)}$ is a groupoid over $M_{(\alpha)}$. Recall that we denote its orbit space by $X_{(\alpha)}$. The next result ensures smoothness of all these objects.

\begin{proposition}\label{prop-Morita.types.are.smooth} The Morita type $M_{(\alpha)}$ is a smooth submanifold of $M$, the groupoid $\G_{(\alpha)}\arrows M_{(\alpha)}$ is a Lie groupoid, and the orbit space $X_{(\alpha)}$ is a smooth manifold.
\end{proposition}

\begin{proof}
Let $x\in M_{(\alpha)}$ and let $g\in \G_{(\alpha)}$. By the Linearization theorem (Theorem \ref{lin}), there are opens $U$ around $\O_x$ in $M$ and $V$ around $\O_x$ in $\N\O_x$ such that $\G_U\cong \N(\G_{\O_x})_V$ (note that $\G_U$ is an open in $\G$ containing $g$). 

The idea of the proof is the following: all the conditions we need to verify are local - that $M_{(\alpha)}$ and $\G_{(\alpha)}$ are submanifolds of $M$ and $\G$, that structure maps restrict to smooth maps and that $s,t:\G_{(\alpha)}\to M_{(\alpha)}$ are submersions. This means that it is enough to check them on the opens $V$ and $\N(\G_{\O_x})_V$, slightly abusing notation by thinking of them as opens around $x$ and $g$. Checking that the multiplication restricts to a smooth map is done in an analogous way as for the other maps, but in a neighbourhood of a composable pair, in the local model.

The situation becomes simple in the linear local model: indeed, $U\cap M_{(\alpha)}$ corresponds to $V\cap (\N\O_x)_{(\alpha)}$; to check that this is a submanifold of $\N\O_x$ we might as well substitute $V$ by its saturation $\tilde{V}$. Using the bundle description $\tilde{V}\cap N(\G_{\O_x})\cong P_x\x_{\G_x} W$ and Lemma \ref{reduceproblem}, we have that $\tilde{V}\cap (\N\O_x)_{(\alpha)}$ is equal to the saturation of $W^{\G_x}$ inside $P_x\x_{\G_x} W$ and so it is a submanifold, diffeomorphic to $\O_x\x W^ {\G_x}$. Its dimension is constant along $M_{(\alpha)}$, so $M_{(\alpha)}$ is a submanifold of $M$.

Moreover, $\G_{(\alpha)}\cap \G_U$ corresponds to $(N(\G_{\O_x})_V)_{(\alpha)}$;
for our purposes (local verifications) we can substitute $V$ by $\tilde{V}$, since $(N(\G_{\O_x})_V)_{(\alpha)}$ is a neighbourhood of $g$ inside of $(N(\G_{\O_x})_{\tilde{V}})_{(\alpha)}$.
Using the description \[N(\G_{\O_x})\cong (P_x\x P_x)\x_{\G_x} \N_x,\] we are interested in the restriction of $(P_x\x P_x)\x_{\G_x} W$ to \[\tilde{V}\cap (N\O_x)_{(\alpha)}\cong \O_x\x W^ {\G_x},\] which is exactly $Gauge(P_x)\x W^{\G_x}\arrows \O_x\x W^ {\G_x}$ - a Lie subgroupoid of $N(\G_{\O_x})$. Therefore, we conclude that on a neighbourhood of $g$, $\G_{(\alpha)}$ is a smooth submanifold of $\G$, that the restriction of all the structure maps of $\G$ (except possibly the multiplication) to it are smooth, and that the restrictions of the source and target are submersions. 

One proceeds in a completely analogous way to check that the restriction of the multiplication is smooth: one only needs to check it a neighbourhood of a composable pair $(g,h)\in \G_{(\alpha)}^{(2)}$; it is enough to work in the local model, where such a neighbourhood can be constructed starting from $N(\G_{\O_x})_V\x N(\G_{\O_x})_V$, since $N(\G_{\O_x})_V$ is a neighbourhood of both $g$ and $h$, and then the restriction of the multiplication coincides with the multiplication of $Gauge(P_x)\x W^{\G_x}$.

This proves that $\G_{(\alpha)}\arrows M_{(\alpha)}$ is a Lie groupoid and therefore Proposition \ref{onemoritatype} tells us that $X_{(\alpha)}$ is a smooth manifold.
\end{proof}


\begin{theorem}[Main Theorem 2]\label{thm-MT2} Let $\G\arrows M$ be a proper Lie groupoid. Then the partitions of $M$ and of $X=M/\G$ by connected components of Morita types are stratifications. 

Given a Morita equivalence between two Lie groupoids $\G$ and $\H$, the induced homeomorphism at the level of orbit spaces preserves the canonical stratification.

\end{theorem}

\begin{proof} Using the local model for a groupoid around a point $x$ in the base, we have seen in Remark \ref{rmk-morita_types_around_point} how to index Morita types in an open around $x$ by the Morita types for the action groupoid associated to the action of $\G_x$ on the linear slice $W$. In such a case, we have seen in Proposition \ref{prop-morita_types_action_groupoid}  that the partition by connected components of Morita types coincides with the canonical stratification $\S_{\G_x}(W)$, which is locally finite. We have also seen in Proposition \ref{prop-Morita.types.are.smooth} that the elements of the partitions by Morita types on $M$ and $X$ are smooth manifolds, with the subspace topology.

Let us check the condition of frontier. Let $T_1$ and $T_2$ be connected components of Morita types $M_{(\alpha)}$ and $M_{(\beta)}$ respectively and suppose that $\overline{T_1}\cap T_2\neq\emptyset$. 

First of all, let us start by checking that the condition of frontier for $T_1$ and $T_2$ holds on a neighbourhood $U$ of any point $x\in\overline{T_1}\cap T_2$, i.e., that $T_2\cap U\subset \overline{T_1}\cap U$. For this we use the local model of $\G$ around the point $x$, as in Proposition \ref{prop-local_model_at_point}.
As explained in Remark \ref{rmk-morita_types_around_point}, the Morita types in such a neighbourhood of $x$ are identified with the product of an open ball $O$ in $\O_x$ with the Morita types for the action groupoid $\G_x\lx W$, where $W$ is an invariant open ball  in $\N_x$; similarly, passing to connected components, $T_i\cap U_0$ is given by $O\x (T_i\cap W)$, for $i=1,2$. So it is enough to check that $y\in \overline{T_1}$ for all points $y\in W\cap T_2$.
We conclude that this holds from the fact that the partition by connected components of Morita types for an action groupoid on the slice $S$ induces the canonical stratification on $S$ associated to the action of $\G_x$.

Let $y$ be any other point in $T_2$ and consider any continuous path $\gamma$ from $x$ to $y$ in $T_2$. 
We can cover the image of $\gamma$ by finitely many open subsets $U_i$ of $M$, centred at points $p_i$ along $\gamma$, such that on each of them, the frontier condition for $T_1$ and $T_2$ holds, because of the same considerations made above for $x$. Hence $y\in \overline{T_1}$ and so $T_2\subset \overline{T_1}$. The continuity of $\pi$ ensures that the condition of frontier also holds on the orbit space.

The statement on Morita invariance follows from the fact that, by definition, the Morita type of $\O\in X$ only depends on Morita invariant information.
\end{proof}

As mentioned before, these stratifications are called the canonical stratifications of $M$ and $X$ and are denoted by $\cS_\G(M)$ and $\cS(X)$.


\subsection{Principal and regular types}\label{subsection_principal}

In this section we discuss principal and regular orbits and give a proof of the Principal type theorem for proper Lie groupoids (Theorem \ref{thm-principal_type}), generalizing the corresponding result for proper Lie group actions.

In analogy with the case of proper actions, we denote by $M^\mathrm{princ}$ the $\cS_\G(M)$-regular part of $M$ and orbits inside $M^{\mathrm{princ}}$ are called \textbf{principal orbits}. Similarly we denote by $X^\mathrm{princ}\subset X$ the $\cS (X)$-regular part of $X$. As before, we get a criteria for when a point $x\in M$ lies in $M^{\mathrm{princ}}$ - combining Lemma \ref{stratif-lemma-0} and the local model of $\G$ around $x$, given as in Proposition \ref{prop-local_model_at_point} by $(O\x O)\x (\G_x\lx W)\arrows O\x W$, we arrive at the condition that $W^{\G_x}$ is open in $W$. This leads to:

\begin{lemma}\label{Gpd-str-prin-pts} For a point $x\in M$, the following are equivalent:
\begin{enumerate}
\item[1.] $x\in M^{\mathrm{princ}}$ 
\item[2.] the action of $\G_x$ on the normal space to the orbit is trivial.
\end{enumerate}
\end{lemma}

However, one difference appears: while for a proper action, all orbits close enough to a principal orbit are diffeomorphic to it, for a proper groupoid this might not be the case. In general, this property holds for the principal orbits of a proper groupoid if and only if the restriction of the groupoid to $M^{\mathrm{princ}}$ is source-locally trivial, which is the case for action groupoids. For example, consider the Lie groupoid associated to a submersion between connected manifolds. It is proper and all orbits are principal, so the similar statement as for proper actions holds if and only if all fibres of the submersion are diffeomorphic.

At this stage we know that $M^{\mathrm{princ}}$ is open, dense, and a union of Morita types (hence invariant). The Principal type theorem tells us a bit more about its geometry, but before we get into it and as preparation for its proof, let us look at an infinitesimal version of the concepts of canonical stratification and principal orbits.

The \textbf{infinitesimal canonical stratification} on $M$ given by a proper groupoid is constructed exactly like the stratification $\cS_\G(M)$, substituting everywhere the isotropy group $\G_x$ and isotropy action on $\N_x$, at a point $x$, by the corresponding isotropy Lie algebra $\g_x$ and induced Lie algebra action of $\g_x$ on $\N_x$. The corresponding stratification is denoted by $\cS^\mathrm{inf}_\G(M)$.

\begin{remark}
At the level of the orbit space we do not have in general an infinitesimal canonical stratification, as the Morita types may fail to be manifolds. They are, however, naturally endowed with an orbifold structure.

Denote by $\mathcal{E}(\mathcal{G})$ be the foliation groupoid associated to a regular Lie groupoid $\mathcal{G}$ over $M$ by dividing out the action of the connected components $\mathcal{G}_x^0$ of the isotropy groups $\mathcal{G}_x$ (for the smoothness of $\mathcal{E}(\mathcal{G})$ see Proposition 2.5 in \cite{moerdijk}). 

The groupoids $\mathcal{G}$ and $\mathcal{E}(\mathcal{G})$ define the same foliation on $M$. If $\mathcal{G}$ is proper, then $\mathcal{E}(\mathcal{G})$ is proper as well \cite{moerdijk}. Therefore, $\mathcal{E}(\mathcal{G})$ defines an orbifold structure on the quotient $M/ \mathcal{G}$. Keep in mind that when seen as differentiable stacks, $M / / \mathcal{E}(\mathcal{G})$ is different from $M/ / \mathcal{G}$ because we lost isotropy information when passing from $\mathcal{G}$ to $\mathcal{E}(\mathcal{G})$ (unless $\mathcal{G}$ was a foliation groupoid to begin with).

Let $M^\mathrm{inf}_{(\alpha)}$ be an infinitesimal Morita type and let $\mathcal{G}^\mathrm{inf}_{(\alpha)}$ denote the restriction of $\mathcal{G}$ to it. Since $\mathcal{G}^\mathrm{inf}_{(\alpha)}$ is proper and regular, $\mathcal{E}(\mathcal{G}^\mathrm{inf}_{(\alpha)})$ defines an orbifold structure on $X^\mathrm{inf}_{(\alpha)}=M^\mathrm{inf}_{(\alpha)} / \mathcal{G}^\mathrm{inf}_{(\alpha)}$.
\end{remark}

The constructions and results are very reminiscent of the ones for proper actions: we denote by $M^\mathrm{reg}$ and $X^\mathrm{reg}$ the $\cS^\mathrm{inf}_\G(M)$-regular part of $M$ and the image of its projection to $X$, respectively; orbits in $M^\mathrm{reg}$ are called \textbf{regular} and the following result characterizes them:

\begin{lemma}\label{Gpd-str-reg-pts} For any point $x\in M$, the following are equivalent:
\begin{enumerate}
\item[1.]  $x\in M^{\mathrm{reg}}$;
\item[2.]  the infinitesimal action of $\gg_{x}$ on the normal space to the orbit is trivial; 
\item[3.]  the action of $G_{x}^{0}$ on the normal space to the orbit is trivial;
\item[4.]  all the orbits through points close to $x$ have the same dimension.
\end{enumerate}
\end{lemma}

The proof of this lemma is straightforward by using Lemma \ref{stratif-lemma-0} and the local model of Proposition \ref{prop-local_model_at_point} for the Lie groupoid $\G$ around $x$.

The main difference with the case of proper actions is that it is no longer true that all the orbits $\O_y$ through points $y$ close to a regular point $x$ are coverings of $\O_x$. It is still true in the case of source-locally trivial groupoids.

From the definitions it is immediate that $M^\mathrm{princ}\subset M^\mathrm{reg}$ and so also that $X^\mathrm{princ}\subset X^\mathrm{reg}$. It is also clear that like for $M^\mathrm{princ}$, it holds that $M^\mathrm{reg}$ is open, dense and consists of a union of Morita types.
We look into connectedness properties for $M^\mathrm{princ}$ and $M^\mathrm{reg}$. With Lemma \ref{stratif-lemma-4} in mind, we start by looking at codimension $1$ Morita types.

\begin{lemma}[Codimension 1 Morita types]\label{codim1}

Let $\G\arrows M$ be a proper Lie groupoid and $x\in M$. Suppose that the codimension of $M_{(x)}$ in $M$ is 1.
Then \begin{enumerate} \item[1.] All orbits in $M_{(x)}$ are regular, i.e., $M_{(x)}\subset M^\mathrm{reg}$;

\item[2.] $F$ is 1-dimensional, where $F$ is an orthogonal complement to $(\N_x)^{\G_x}$ in $\N_x$, for any $\G_x$-invariant inner product on $\N_x$, as in Lemma \ref{reduceproblem}; $\G_x$ acts non-trivially on $F$ as $O(1,\R)=\Z_2$.

\item[3.] The orbits near $\O_x$ which are not in $M_{(x)}$ are principal.

\item[4.] There is a neighbourhood $U$ of $x$ in $M$ such that the intersection of each orbit which is not in $M_{(x)}$ with $U$ is a two-fold covering of $\O_x\cap U$.
\end{enumerate}
\end{lemma}

\begin{proof}
We can use Lemma \ref{lemma-reducemorita} to restrict our attention to a slice at $x$. We have that $\dim M_{(x)}=\dim \O_x + \dim \N_x^{\G_x}$. In the case that the codimension of $M_{(x)}$ is $1$, $F$ must be $1$-dimensional. Now $\G_x$ acts non-trivially on $F$ by orthogonal transformations, so it must act as multiplication by $\{+1,-1\}$. This proves part 2 of the lemma. 

Since the orbits near $\O_x$ are equal to the saturation of their intersection of a slice, which can be identified with $\N_x$, then the dimension of the orbits is constant near $\O_x$, proving part 1 of the lemma. 

For part 3, note that the orbits near $\O_x$ which are not in $M_{(x)}$ are the saturation of elements $v\in F\backslash \{0\}$. All such points have the same Morita type, so they all belong to the same open Morita type and therefore to $M^\mathrm{princ}$. 
Part 4 is then an easy consequence of the local model for $\G$ around $x$ (Proposition \ref{prop-local_model_at_point}). 
\end{proof}

\begin{corollary} Let $\G$ be a Lie groupoid over $M$. Then

\begin{itemize}
\item[1.] $\cS_{\G}^{\mathrm{inf}}(M)$ does not contain codimension 1 strata.
\item[2.] $M^{\mathrm{reg}}$ is not only dense and open in $M$, but also connected. 
\item[3.] hence $\cS_{G}^{\mathrm{inf}}(M)$ has one and only one maximal stratum. 
\end{itemize}
\end{corollary}
\begin{proof} Part 1 is a reformulation of part 1 of Lemma \ref{codim1}; part 2 follows from part 1 and Lemma \ref{stratif-lemma-4}; part 3 is just a reformulation of part 2.
\end{proof}

Unlike the stratification $\cS^\mathrm{inf}_\G(M)$, the canonical stratification $\cS_\G(M)$ can have codimension $1$ strata. However, Lemma \ref{codim1} also implies that on a small neighbourhood of a regular point $x$ belonging to a codimension $1$ stratum, all the orbits not contained in this stratum belong to $M^\mathrm{princ}$; the stratum through $x$ then disconnects $M^\mathrm{princ}$ into two half-spaces, which are permuted by the isotropy action. In other words, passing to the orbit space we conclude that a neighbourhood of $\O_x$ in $X$ looks like a neighbourhood of the boundary point $\O_x$ on a manifold with boundary.

\begin{example}
A simple illustration of this behaviour is the following.
Consider an action of the circle on the M\"obius band (and its associated action groupoid), so that all orbits are regular and only the central orbit is not principal - its Morita type consists of only itself and has codimension 1. We are then in the conditions of Lemma \ref{codim1} and it is clear that the orbit space is a manifold with boundary: a closed half-line.
\end{example}

\begin{theorem}[Principal Morita type]\label{thm-principal_type} Let $\G\arrows M$ be a proper Lie groupoid. Then $X^\mathrm{princ}$ is not only open and dense but also connected, hence $\cS(X)$ has a single maximal stratum. 
\end{theorem}

\begin{proof}
Let $x$ and $y$ be points on two principal orbits and consider a path $\gamma$ in $M$ connecting them. Since the image of $\gamma$ is compact, by local finiteness of the stratification, it meets finitely many strata. As a consequence, we can use a version the Transversality homotopy theorem \cite[p.\ 72]{GP} to obtain a curve $\gamma'$, between $x$ and $y$, which is transverse to all strata. This means that it misses all strata of codimension $2$ or higher, and it meets finitely many strata of codimension $1$ transversally. 

If $p$ is a point in the intersection of a codimension $1$ stratum with the image of $\gamma'$, we know by Lemma \ref{codim1} that a neighbourhood $U$ of $\O_p$ in $X$ looks like a neighbourhood of a boundary point in a closed half-space, in which $\O_p$ is a boundary point (and all orbits not on the boundary are principal).
Since $\gamma'$ is transverse to the stratum of $p$, we can assume that the part of the image of $\pi\circ \gamma'$ that lies in $U$ only touches the boundary at $\O_p$, so we can homotope $\pi\circ \gamma$ so that its image in $U$ misses the boundary completely, hence it is contained in $X^\mathrm{princ}$. Doing this to the curve $\pi\circ \gamma'$ in $X$ finitely many times, we obtain a new curve connecting $\O_x$ to $\O_y$, completely contained in $X^\text{princ}$.
\end{proof}

In the case of the action groupoid associated to a proper Lie group action, Theorem \ref{thm-principal_type} becomes the classical Principal orbit type theorem mentioned in Section \ref{sec-proper_actions_strat}. With some mild assumptions we obtain a consequence about $M^\mathrm{princ}$ as well:

\begin{corollary}If $\G\arrows M$ is a proper Lie groupoid with connected orbits (for example, if it has connected $s$-fibres), then $M^\mathrm{princ}$ is connected, hence $\cS_\G(M)$ has a single maximal stratum.
\end{corollary}

\section{Orbispaces as differentiable stratified spaces}
\label{chp5-smooth_stratifications}

When a space $X$ admits both a stratification and a smooth structure, it is natural to ask how the strata fit together with respect to the smooth structure. For example, if $X$ is a smooth manifold, it is natural to require that the strata are submanifolds.

In general, one way to proceed would be to think of $X$ as a stratified space and then give it some smooth structure in a way compatible with the stratification, for example via an appropriate atlas. This approach is studied in detail in \cite{Pflaum}. Alternatively, one could think of $X$ as if it was a ``smooth'' space (for example by giving it one of the structures described in Section \ref{chp3-smooth structures}) and then consider a stratification on $X$ that respects the smooth structure. We follow this second approach.

\begin{definition}A \textbf{differentiable stratified  space} consists of a differentiable space $(X,\O_X)$ together with a stratification $\cS$ on $X$ such that the inclusion of each stratum is an embedding of differentiable spaces.
\end{definition}

As remarked in \cite{hessel}, the notion of a reduced differentiable stratified  space is the same as that of a stratified space with smooth structure of Pflaum, defined in terms of singular charts (cf. \cite[Sec.\ 1.3]{Pflaum}).

For a general stratification, the only condition we have required on how the strata should fit together is the frontier condition. In the presence of smooth structure, there are several other conditions that are often imposed, the most common of them being Whitney's conditions (A) and (B).

\begin{definition} Let $M$ be a smooth manifold and let $\cS$ be a stratification on $M$. We say that a pair of strata $(R,S)$ satisfies the \textbf{Whitney condition (A)} if the following condition holds:
\begin{enumerate}
\item[(A)] Let $(x_n)$ be any sequence of points in $R$ such that $x_n$ converges to a point $x\in S$ and $T_{x_n}S$ converges to $\tau\subset T_x M$ in the Grassmannian of $(\dim R)$-dimensional subspaces of $TM$. Then $T_x S\subset \tau$.
\end{enumerate}
 
Let $\phi:U\rmap \R^n$ be a local chart around $x\in S$. We say that $(S,T)$ satisfies the \textbf{Whitney condition (B)} at $x$ with respect to the chart $(U,\phi)$ if satisfies the following condition:

\begin{enumerate}
\item[(B)] Let $(x_n)$ be a sequence as in $(A)$, with $x_n\in U\cap R$ and let  $(y_n)$ be a sequence of points of $U\cap S$, converging to $x$, such that the sequence of lines \[\overline{\phi(x_n)\phi(y_n)}\] converges in projective space to a line $\ell$. Then $(d_x\phi)^{-1}(\ell)\subset \tau$. 
\end{enumerate} 

The pair $(S,T)$ is said to satisfy the Whitney condition $(B)$ if the above condition holds for any point $x\in S$ and any chart around $x$.
\end{definition}

Although we have used charts in the definition of Whitney condition $(B)$, the condition is actually independent of the chart chosen - if $(S,T)$ satisfy Whitney $(B)$ with respect to a chart around $x$, they do so with respect to any chart as well \cite[Lemma 1.4.4]{Pflaum}. 

\begin{remark}\label{rmk-whitney} Note that Whitney's conditions are local, by definition, and they also make sense for a stratification on a subspace of $\R^n$. In this way we can make sense of when a reduced differentiable stratified  space $X$ satisfies Whitney's conditions: by Proposition \ref{prop-reduced-affine} such a space is always covered by open subspaces $(U_i,\O_{X|U_i})$ isomorphic to $(Z,\C^\infty_Z)$, for some closed subset $Z$ of $\R^n$. The subspace $Z_i$ has a decomposition $\cP_i$ induced by the decomposition on $U_i$ (By restricting to an open, we might not have a stratification, but only a decomposition; that is not a problem since all that we want is to check that Whitney's conditions hold for the pieces of the decomposition on $Z_i$ corresponding to the strata on $X$).
\end{remark}

\begin{definition}Let $(X,\O_X, \cS)$ be a reduced differentiable stratified  space. We say that $\cS$ is a \textbf{Whitney stratification} if $X$ is covered by opens $U_i$, such that, in the notation of Remark \ref{rmk-whitney} above, $(U_i,\O_{X|U_i})$ is isomorphic to $(Z,\C^\infty_{Z_i})$,  and all the induced stratifications $\cS_i$ satisfy Whitney's conditions $(A)$ and $(B)$.
\end{definition}

The idea is that Whitney's conditions are local conditions about how the strata fit together that permit drawing important global information about the stratification. Let us mention one particularly important example of this idea: stratifications satisfying Whitney's conditions are locally trivial (see Proposition \ref{prop-loc-triv-strat} below).

\begin{definition}
A \textbf{morphism of stratified spaces} is a continuous map $f:X\rmap Y$ between stratified spaces with the property that for every stratum $S$ of $X$ there is a stratum $R_S$ of $Y$ such that $f(S)\subset Y$ and the restriction $f_{|S}:S\rmap R_S$ is smooth. 
\end{definition}

Let $X_1$ and $X_2$ be two topological spaces, with stratifications $\cS_i$ on $X_i$, $i=1,2$. Then the products of the form $S\x R$ with $S\in \cS_1$ and $R\in\cS_2$ form a stratification on $X_1\x X_2$.

\begin{definition}
A stratification $\cS$ on a space $S$ is called \textbf{topologically locally trivial} if for every $x\in X$ there is an open neighbourhood $U$ of $x$ in $X$, a stratification $\S_F$ on a space $F$, a point $0\in F$ and an isomorphism of stratified spaces   \[\phi: (S\cap U)\x F\rmap U, \] where $S$ is the stratum of $\S$ containing $x$, such that the stratum of $\S_F$ containing  $0$ is simply $\{0\}$, and such that $\phi(y,0)=y$ for any $y\in S\cap U$. In this case $F$ is called the \textbf{typical fibre} over $x$. When $F$ is a cone, $F=CL$, we say that $L$ is the \textbf{link} of $x$.

If $L$ is locally trivial with cones as typical fibres, and that holds again for the links in the points of $L$, and so on, we say that $(X,\cS)$ is a \textbf{cone space}.
\end{definition}

\begin{proposition}\label{prop-loc-triv-strat}[Thom-Mather \cite{Mather2}, \cite{thom}] Any Whitney stratified reduced differentiable space is locally trivial, with cones as typical fibres.
\end{proposition}

\subsection{Orbispaces as differentiable stratified  spaces}

We now focus on the differentiable aspects of the canonical stratifications associated to a proper groupoid.

\begin{proposition}[Main Conclusion] Let $\G\arrows M$ be a proper Lie groupoid. Then $M$ and the orbit space $X=M/\G$, together with the canonical stratifications, are differentiable stratified  spaces. Moreover, the canonical stratifications of $M$ and $X$ are Whitney stratifications.

Any Morita equivalence between two proper Lie groupoids induces an isomorphism of differentiable stratified  spaces between their orbit spaces.
\end{proposition}

\begin{proof}
We have seen before that each connected component of a Morita type is a submanifold of $M$ so, together with its canonical stratification, $M$ is a differentiable stratified  space. Moreover, using the local model of $\G$ around a point in $M$, as in Remark \ref{rmk-morita_types_around_point}, we see that since the stratification of the action of $\G_x$ on $\N_x$ is a Whitney stratification, the same holds for the canonical stratification on $M$.

Now let us focus on the orbit space. We already know that the strata are locally closed subspaces of $X$. It is a local problem to verify that they are embedded in $X$ as differentiable spaces and satisfy Whitney's conditions. Let $x\in X$ and let $U$ be a neighbourhood of $x$ in $X$ for which the intersection with the canonical stratification coincides with the canonical stratification on $U$ associated with a compact Lie group representation (again, by the local description of Morita types as in Remark \ref{rmk-morita_types_around_point})

The statement on the Whitney conditions for the orbit space now follows from a result of Bierstone \cite{bierstone}, which states that the orbit space of a representation of a compact group has a Whitney stratification (which coincides with the canonical stratification). 

Morita invariance of the differentiable stratified  space structure on the orbit space $X$ follows from the Morita invariance of the smooth structure \ref{Xissmooth} and of the canonical stratification  \ref{thm-MT2} on $X$.
\end{proof}

\begin{corollary} Let $\G\arrows M$ be a proper Lie groupoid. Then the canonical stratifications on $M$ and $X$ are locally trivial with cones as typical fibres. Moreover, $M$ is a cone space.
\end{corollary}
\begin{proof}The first statement is  a direct consequence of Proposition \ref{prop-loc-triv-strat}. But actually, we already had an explicit description of this fact, by combining Lemma \ref{reduceproblem} and Remark \ref{rmk-morita_types_around_point}. For the stratification on $M$ we find that the typical fibre over $x$ is the orthogonal complement $F$ to $\N_x^{\G_x}$ with respect to an invariant inner product for the $\G_x$-action on $\N_x$; the link of $x$ is the unit sphere $\Sigma$ of $F$. Since the stratification on $\Sigma$ is the canonical stratification for the $\G_x$-action, it has the same properties, so $M$, with the canonical stratification, is a cone space.
\end{proof}

With the notation of the proof of the previous corollary, on $X$ we have that the typical fibre over the orbit $\O_x$ is $F/\G_x$ and the link of $\O_x$ is $\Sigma/\G_x$.

\begin{remark} The proof that the canonical stratifications on $M$ and $X$ are Whitney stratifications, using the language of germ-stratifications, appeared in \cite{hessel}. The authors of \textit{loc. cit.} also use the canonical stratification to prove that the orbit space of a proper groupoid can be triangulated, and a deRham theorem for the basic cohomology of a proper Lie groupoid.

Another example of a differentiable stratified  space arising in the study of proper Lie groupoids is that of the inertia groupoid of a proper Lie groupoid \cite{farsi1,farsi2}.
\end{remark}

\subsection{Orbispaces as stratified subcartesian spaces}

We now briefly discuss stratifications on subcartesian spaces (see Section \ref{sec-sikorski}) coming from orbits of families of vector fields and their relevance for orbispaces.

\begin{definition}
Let $(S,\F_S)$ be a subcartesian space and let $X$ be a derivation of $\F_S$. An \textbf{integral curve} of $X$ through a point $x$ of $S$ is a curve $c:I\rmap S$ such that $I$ is an interval containing $0$ and \[\frac{d}{d\epsilon}f(c(\epsilon))=X(f)(c(t))\] for all $f\in \F_S$ and all $\epsilon\in I$. If the domain of $c$ is maximal with this property, we say that $c$ is a \textbf{maximal integral curve}. By convention, the map $c:\{0\}\to S,\ 0\mapsto x$ is an integral curve of any derivation. 
\end{definition}

\begin{theorem}[Theorem 3.2.1 in \cite{Sniat-book}] Let $(S,\F_S)$ be a subcartesian space and let $X$ be a derivation of $\F_S$. Then for any $x$ in $S$, there is a unique maximal integral curve $c$ of $X$ through $x$ such that $c(0)=x$.
\end{theorem}

Although derivations of $\F_S$ always admit maximal integral curves, they may fail to induce local one-parameter groups of local diffeomorphisms. The ones that do are called vector fields:

\begin{definition}
Let $(S, \F_S)$ be  a subcartesian space. A \textbf{vector field on $S$} is a derivation $X$ of $\F_S$ such that for every $x\in S$, there exists a neighbourhood $U$ of $x$ in S and $\epsilon>0$ such that for all $t\in(-\epsilon,\epsilon)$, the map $\exp(tX)$ is defined on $U$ and its restriction to $U$ is a diffeomorphism of $U$ with an open subset of $S$. 
\end{definition}

\begin{definition}
Let $(S, \F_S)$ be a subcartesian space, let $\mathfrak{F}$ be a family of vector fields on $S$ and let $x\in S$. The \textbf{orbit of the family $\mathfrak{F}$ through $x$} is the set of all points $y\in S$ of the form \[y=\exp(t_nX_n)\circ\ldots\circ\exp(t_1X_1)(x),\] for some $t_1,\ldots,t_n\in \R$ and some $X_1,\ldots,X_n\in \mathfrak{F }$.
\end{definition}

\begin{theorem}[Theorem 3.4.5 in \cite{Sniat2}] Any orbit $O$ of a family $\mathfrak{F}$ of vector fields on a subcartesian space $S$ is a smooth manifold. Moreover, in the topology of $O$ given by the manifold structure, the Sikorski structure induced on $O$ by the inclusion on $S$ coincides  with the manifold structure.
\end{theorem}

\begin{proposition}[Proposition 4.1.2 in \cite{Sniat-book}] The partition of a subcartesian space $S$ by orbits of the family of all vector fields on $S$ satisfies the frontier condition.
\end{proposition}

\begin{corollary}[Corollary 4.1.3 in \cite{Sniat-book}] The partition of a subcartesian space $S$ by orbits of the family of all vector fields on $S$ is a stratification if and only if the partition is locally finite and the orbits are locally closed.
\end{corollary}

An important example of a case where orbits of all vector fields do indeed define a stratification is that of orbit spaces of proper actions.

\begin{theorem}[Theorem 4.3.10 in \cite{Sniat-book}] Let $G$ be a Lie group acting properly on a manifold $M$. The partition of the subcartesian space $M/G$ by orbits of the family of all vector fields on $M/G$ is the canonical stratification $\cS(M/G)$.
\end{theorem}

Note that this means that for the orbit space of a proper Lie group action, the canonical stratification is completely determined by the smooth structure on $M/G$.
Since the orbit space of a proper Lie groupoid is locally diffeomorphic (as a subcartesian space) to the orbit space of a representation of a compact Lie group, we obtain the following (cf. Theorem 4.14 in \cite{Watts2}).

\begin{corollary} Let $X$ be the orbit space of a proper Lie groupoid. Then the canonical stratification $\cS(X)$ coincides with the partition by the orbits of the family of all vector fields on $X$, seen as a subcartesian space. 
\end{corollary}

In the case of a classical orbifold, i.e., one which is presented by an effective proper groupoid (cf. e.g. \cite{ieke_mrcun}), the following result of Watts shows that we can actually recover all the information from the smooth structure:

\begin{theorem}[Main theorem in \cite{Watts2}] Given an effective orbifold $X$, an orbifold atlas for it can be constructed out of invariants of the ring of smooth functions $C^\infty(X)$.
\end{theorem}

\bibliography{mybiblio}{}
\bibliographystyle{plain}

\end{document}